\documentclass[11pt]{article}
\usepackage{amsmath,enumerate,amsfonts,color,amssymb,amsthm}

\usepackage{tikz}

\def\mcA{{\mycal A}}
\def\mcB{{\mycal B}}

\def\mcF{{\mycal F}}

\def\mcP{{\mycal P}}

\def\mcR{{\mycal R}}

\def\eps{{\varepsilon}}

\newtheorem{theorem} {\sc  Theorem\rm} [section]

\newtheorem{corollary} [theorem] {\sc  Corollary\rm}
\newtheorem{lemma} [theorem] {\sc  Lemma\rm}
\newtheorem{proposition} [theorem] {\sc  Proposition\rm}
\newtheorem{prop} [theorem] {\sc  Proposition\rm}

\newtheorem{remark}[theorem]{\sc  Remark\rm}

\newcounter{marnote}

\DeclareFontFamily{OT1}{rsfs}{}
\DeclareFontShape{OT1}{rsfs}{m}{n}{ <-7> rsfs5 <7-10> rsfs7 <10-> rsfs10}{}
\DeclareMathAlphabet{\mycal}{OT1}{rsfs}{m}{n}

\def\dist{{\rm dist}}
\def\tr{{\rm tr}}

\def\mcS{{\mycal{S}}}

\def\be{\begin{equation}}
\def\ee{\end{equation}}

\def\Id{{\rm Id}}
\def\dist{{\rm dist}}
\def\tr{{\rm tr}}

\def\mcS{{\mycal{S}}}

\newcommand{\R}{\mathbb{R}}

\newcommand{\N}{\mathbb{N}}

\def\be{\begin{equation}}
\def\ee{\end{equation}}
\def\bea#1\eea{\begin{align}#1\end{align}}
\def\non{\nonumber}
\def\mcR{{\mycal{R}}}

\numberwithin{equation}{section}

\usepackage[english]{babel}
\usepackage{enumitem,esint,graphicx,subcaption}

\renewcommand{\d}{\mathrm{d}}

\newcommand{\SO}{\mathrm{SO}}
\newcommand{\abs}[1]{\left| #1 \right|} 
\newcommand{\norm}[1]{\left\| #1 \right\|}

\begin{document}
\title{Polydispersity and surface energy strength in  nematic colloids}

\author{Giacomo Canevari\thanks{Dipartimento di Informatica,
Universit\`a degli Studi di Verona,
Strada le Grazie 15, 37134 Verona, Italy
(giacomo.canevari@univr.it)}\, and 
Arghir Zarnescu\thanks{IKERBASQUE, Basque Foundation for Science, Maria Diaz de Haro 3,
48013, Bilbao, Bizkaia, Spain.}\,  
\thanks{BCAM,  Basque  Center  for  Applied  Mathematics, 
Mazarredo  14,  E48009  Bilbao,  Bizkaia,  Spain. }\,
\thanks{``Simion Stoilow" Institute of the Romanian Academy, 
21 Calea Grivi\c{t}ei, 010702 Bucharest, Romania.}} 

\maketitle

\begin{abstract} 
 We consider a Landau-de Gennes model for a polydisperse, inhomogeneous suspension of 
 colloidal inclusions in a nematic host, in the dilute regime. 
 We study the homogenised limit and compute the effective free energy of the composite material.
 By suitably choosing the shape of the inclusions and 
 imposing a quadratic, Rapini-Papoular type surface anchoring energy density,
 we obtain an effective free energy functional with an additional linear term, 
 which may be interpreted as an ``effective field'' induced by the inclusions. 
 Moreover, we compute the effective free energy
 in a regime of ``very strong anchoring'', that is,
 when the surface energy effects dominate over the volume free energy.
\end{abstract}




\section{Introduction}

We consider a mixture of mesoscale size particles within an ambient fluid that contains locally aligned microscopic scale rod-like molecules, that is a nematic liquid crystals. This type of mixture, which is usually referred to as a \emph{nematic colloid} material, has emerged in the recent years as the material of choice for testing a number of exciting hypothesis in the design of new materials. An overview of the field, and its applications, from the physical point of view is available in the reviews \cite{Lavrentovich,Smalyukh}.

The mathematical studies of such systems are still relatively few and focus on two extreme situations: 
\begin{itemize}
\item the effect produced by one colloidal particle, particularly related to the so-called {\it `defect patterns'} that is the strong distortions produced at the interface between  the particle and the ambient nematic fluid;
\item the collective effects produced by the presence of many particle, fairly uniformly distributed, with a focus on the homogenised material.
\end{itemize}

In the first direction one should note that defects appear because of the 
anchoring conditions at the boundary of the particles,
which generate topological obstructions \cite{ABL1, ABL2, CRamaswamyMajumdar,
CSegattiVeneroni, CSegatti, CSegatti-proc, WangCMajumdar}. 
There  there have been a number of works, identifying several 
physically relevant regimes \cite{ABL1} and the influence of external fields \cite{ABL2}. 

Our work focuses on the second direction, namely on long-scale effects produced by the effects of a large number of particles, namely on the homogenisation regime. There a couple of works in this direction, on which our work builds, namely \cite{multiscale, BCD, ferromagnetic}.  The main novelty of our approach,
compared to those in~\cite{multiscale, BCD, ferromagnetic}, is that we allow for a much larger class 
of surface energy densities: we do not assume that the surface energy density is bounded from below and
we do consider surface energy densities of quartic growth, which is the maximal growth compatible with the Sobolev embeddings. Surface energy densities of quartic growth have been 
proposed in the physical literature~\cite{Goossens, SluckinPon, genpot, Rey}.

We focus on a regime in which the total volume of the particles is much smaller than that of the ambient nematic environment, that is a  dilute regime. Our aim is to provide a mathematical understanding of statements from the physical literature e.g. \cite{Reznikov,Fenghua} showing that in such a regime the colloidal nematics behave like a {\it homogenised, standard nematic material}, but with  different  (better) properties than those of the original nematic material.

In our previous work~\cite{CanevariZarnescu}, we provided a first approach to these issues and we showed that using periodically distributed {\it identical} particles, one can design a suitable surface energy to obtain an apriori designed potential, that models the main physical properties of the material (in particular the nematic-isotropic transition temperature).

The purpose of these notes is two-fold, aiming to understand what happens when one goes beyond some of the restrictions imposed in our previous work~\cite{CanevariZarnescu}.
The first goal is to extend the main results of~\cite{CanevariZarnescu}
to \emph{polydisperse} and \emph{inhomogenoeus} nematic colloids; the second
one, is to explore a regime of parameters that differs 
from the one considered in~\cite{CanevariZarnescu}.

Realistically, a set of colloidal inclusions will hardly be identical: 
the particles will differ in their size, shape, or charge.
In order to account for polydispersity, we will consider 
several populations of colloidal inclusions, which may differ in their shape
and properties. Moreover, we will not require the centres of mass of the inclusions
to be homogeneously distributed in space. 
In mathematical terms, let~$\mcP^1$, $\mcP^2$, \ldots, $\mcP^J$ be subsets of~$\R^3$
(the reference shapes of the inclusions), and let~$\Omega\subseteq\R^3$ 
be a bounded, smooth domain (the container). We define
\begin{equation} \label{inclusions}
 \mcP_\eps^j := \bigcup_{i = 1}^{N_\eps^j}\left(x^{i,j}_\eps + \eps^\alpha
 R_\eps^{i,j}\mcP^j\right) \qquad \textrm{for } j\in\{1, \, \ldots, \, J\},
\end{equation}
where~$\alpha$ is a positive number, the $x_\eps^{i,j}$'s are points 
in~$\Omega$ and the~$R_\eps^{i,j}$ are rotation matrices that
satisfy suitable assumptions (see Section~\ref{Sec:main}).
As in~\cite{ferromagnetic, CanevariZarnescu}, we work in the \emph{dilute regime}, namely
we assume that~$\alpha > 1$ so that the total volume occupied by the inclusions,
$\abs{\mcP_\eps}\approx \eps^{3\alpha - 3}$, tends to zero as~$\eps\to 0$.
However, we also assume that $\alpha < 3/2$ so that the total surface area of the inclusions,
$\sigma(\partial\mcP_\eps)\approx\eps^{2\alpha-3}$,
diverges as~$\eps\to 0$.
We define $\mcP_\eps:= \cup_{j} \mcP_{\eps}^j$ 
and~$\Omega_\eps := \Omega\setminus\mcP_\eps$ (the space 
that is effectively occupied by the nematic liquid crystal).
In accordance with the Landau-de Gennes theory,
the nematic liquid crystal is described by a tensorial order parameter,
that is, a symmetric, trace-free $(3\times 3)$-matrix field~$Q$.
We consider the free energy functional
\begin{equation} \label{LdG-intro}
 \mcF_\eps[Q] := \int_{\Omega_\eps} \left(f_e(\nabla Q) + f_b(Q)\right)\d x +
 \eps^{3-2\alpha} \sum_{j=1}^J \int_{\partial\mcP_\eps^j} f_s^j(Q, \, \nu) \, \d\sigma.
\end{equation}
Here, $f_e$, $f_b$ are suitable elastic and bulk energy densities
(in the Landau-de Gennes theory, $f_e$ is typically a positive definite, 
quadratic form in~$\nabla Q$ and~$f_b$ is a quartic polynomial in~$Q$; 
see Section~\ref{Sec:main}), $f_s^j$ is a surface anchoring energy densities
(which may vary for different species of inclusions),
and~$\nu$ denotes the exterior unit normal to~$\Omega$.
We prove a convergence result for \emph{local} minimisers of~$\mcF_\eps$
to local minimsers of the effective free energy functional:
\begin{equation*}
 \mcF_0[Q] := \int_{\Omega} \left(f_e(\nabla Q) + f_b(Q) + f_{hom}(Q, \, x)\right)\d x.
\end{equation*}
The ``homogenised potential'' $f_{hom}$, which keeps memory of the surface integral,
is explicitly computable in terms of the~$f_s^j$'s, the distribution
of the centres of mass~$x_\eps^{i,j}$ and the rotations~$R_\eps^{i,j}$.
As an application of this result, we show that polydisperse inclusions 
may be used to mimic the effects of an applied electric field. More precisely,
for a pre-assigned parameter~$W\in\R$ and a pre-assigned symmetric matrix~$P$,
we may tune the shape~$\mcP_\eps^j$ of the inclusions and
the surface energy densities, so to have in the limit
\[
 f_{hom}(Q) = W \tr(QP).
\]
When~$P$ has the form~$P = E\otimes E$ for some~$E=\R^3$,
this expression may be interpreted as an
electrostatic energy density induced by the 
``effective field''~$E$,
up to terms that do not depend on~$Q$.

\bigskip Moving beyond the issue of polydispersity we consider another physically restrictive assumption we made in \cite{CanevariZarnescu}, namely concerning the anchoring strength.
In~\eqref{LdG-intro}, the scaling of parameters is chosen so 
to have a factor of~$\eps^{3-2\alpha}$ in front of the surface integral,
which compensates exactly the growth of the surface area, 
$\sigma(\partial\mcP_\eps)\approx\eps^{2\alpha - 3}$.
However, other choices of the scaling are possible, corresponding 
to different choices of the anchoring strength at the 
boundary of the inclusions. One can easily check that having a weaker anchoring, say of the type $\eps^{2\alpha-3+\delta}$ with $\delta>0$ will lead to a vanishing of the homogenized term, so the main interest is to understand what happens for stronger anchoring.To illustrate this possibility,
we study the asymptotic behaviour, as~$\eps\to 0$, of minimisers of 
\begin{equation*} 
 \mcF_{\eps,\gamma}[Q] := \int_{\Omega_\eps} 
 \left(f_e(\nabla Q) + f_b(Q)\right)\d x +
 \eps^{3-2\alpha-\gamma} \sum_{j=1}^J \int_{\partial\mcP_\eps^j} f_s^j(Q, \, \nu) \, \d\sigma,
\end{equation*}
where~$\gamma$ is a positive parameter. This scaling
corresponds to a much stronger surface anchoring and we expect the 
behaviour of minimisers to be dominated by the surface
energy, as~$\eps\to 0$. Indeed, we will show
that for $\gamma$ small enough the functionals~$\mcF_{\eps, \,\gamma}$ 
$\Gamma$-converge to the constrained functional
\[
 \widetilde{\mcF}(Q) := \begin{cases}
                        \displaystyle\int_{\Omega} 
                          \left(f_e(\nabla Q) + f_b(Q)\right)\d x 
                              & \textrm{if } f_{hom}(Q(x), \, x) = 0
                              \textrm{ for a.e. } x\in\Omega \\
                          +\infty & \textrm{otherwise,}
                       \end{cases}
\]
as~$\eps\to 0$. 

This result leaves a number of interesting of open problems, the most immediate ones being what is the optimal range of $\gamma$ for which this holds and, directly related to this, if one gets a different limit for large values of $\gamma$. 

The paper is organized as follows: in the following, in Section~\ref{Sec:main} we consider the polydisperse setting and the general homogenisation result. The main results of this section, namely Theorem~\ref{th:loc-min} and Proposition~\ref{prop:liminf} are presented after the introduction of the mathematical setting, in Subsection~\ref{Sec:main}. The proof of the results is provided in Subsection~\ref{subsec:proofhom} after a number of preliminary results, need in the proof, provided in Subsection~\ref{subsec:prelim}.

In Section~\ref{sec:linear} we provide an application of the results in Section~\ref{Sec:main}, namely showing in Proposition~\ref{prop:homlinear} that one can in a polydisperse setting obtain a linear term in the homogenised potential.

Finally, in Section~\ref{sect:gamma} we study the case when the scaling of the anchoring strength is $\eps^{3-2\alpha-\gamma}$ with $\gamma$ suitably small, and provide in Theorem~\ref{th:gamma} the $\Gamma$-converegence result mentioned above. Its proof is done at the end of the section after a number of preliminary results.

\section{An homogenisation result for polydisperse, inhomogeneous 
nematic colloids in the dilute regime}

\subsection{Statement of the homogenisation result}
\label{Sec:main}

\paragraph*{The Landau-de Gennes $Q$-tensor.}

In the Landau-de Gennes theory, the local configuration 
of a nematic liquid crystal is represented by a symmetric,
symmetric, trace-free, real $(3\times 3)$-matrix, known as the $Q$-tensor,
which describes the anisotropic optical properties of the 
medium~\cite{deGennes, Mottram}. We denote by~$\mcS_0$ the set of matrix as above.
For~$Q$, $P\in\mcS_0$, we denote~$Q\cdot P := \tr(QP)$. This defines a 
scalar product on~$\mcS_0$, and the corresponding norm will be denoted by
$\abs{Q} := (\tr(Q^2))^{1/2} = (\sum_{i,j} Q_{ij})^{1/2}$.

\paragraph*{The domain.}

let~$\mcP^1$, $\mcP^2$, \ldots, $\mcP^J$ be subsets of~$\R^3$
(the reference shapes of the inclusions), and let~$\Omega\subseteq\R^3$ 
be a bounded, smooth domain (the container). We define~$\mcP_\eps^j$
as in~\eqref{inclusions}, where~$\alpha$, $x_\eps^{i,j}$, 
$R_\eps^{i,j}$ satisfy the following assumptions:
\begin{enumerate}[label=(H\textsubscript{\arabic*}), ref=H\textsubscript{\arabic*}]
 \item \label{hp:alpha} \label{hp:first} $1 < \alpha < 3/2$.
 
 \item \label{hp:Omega_eps}
 There exists a constant $\lambda_\Omega>0$ such that
 \[
  \dist(x_\eps^{i,j}, \, \partial\Omega) + 
  \frac{1}{2} \inf_{(h, \, k)\neq (i, j)}
  |x_\eps^{h, k} - x_\eps^{i,j}| \geq \lambda_\Omega\eps
 \]
 for any~$\eps>0$, any~$j\in\{1, \, \ldots, \, J\}$  and
 any~$i\in\{1, \, \ldots, \, N_\eps^j\}$.
 
 \item \label{hp:conv} For any~$j\in\{1, \, \ldots, \, J\}$,
 there exists a non-negative function $\xi^j\in L^\infty(\Omega)$,
 such that
 \[
  \mu_\eps^j := \eps^3\sum_{i=1}^{N_\eps^j}\delta_{x_\eps^{i, j}}
  \rightharpoonup^* \xi^j \, \d x
  \qquad \textrm{as measures in } \R^3\textrm{, as } \eps\to 0.
 \]
 
 \item \label{hp:R} For any~$j\in\{1, \, \ldots, \, J\}$, 
 there exists a Lipschitz-continuous map~$R_*^j\colon\overline{\Omega}\to\SO(3)$
 such that $R_\eps^{i,j} = R_*^j(x_\eps^{i,j})$ for any~$\eps>0$ 
 and any~$i\in\{1, \, \ldots, \, N_\eps^j\}$.
 
 \item \label{hp:P} For any~$j\in\{1, \, \ldots, \, J\}$, 
 $\mcP^j\subseteq\R^3$ is a compact, convex set 
 whose interior contains the origin.
\end{enumerate}

The assumption~\eqref{hp:Omega_eps} is a separation condition
on the inclusions. As a consequence of~\eqref{hp:Omega_eps}, the number of
the inclusions, for each population~$j$, is $N_\eps^j\lesssim\eps^{-3}$.
Therefore, the total volume of the inclusions in each population
is bounded by $N_\eps^j\eps^3\lesssim \eps^{3\alpha-3}\to 0$, because of~\eqref{hp:alpha}.
Thus, we are in the diluted regime, as in \cite{ferromagnetic, CanevariZarnescu}. 
We define
\[
 \mcP_\eps := \bigcup_{j=1}^J\mcP_\eps^j, \qquad
 \Omega_\eps := \Omega\setminus\mcP_\eps.
\]
The assumption~\eqref{hp:P} guarantees that $\Omega_\eps$ is a Lipschitz domain.

\paragraph*{The free energy functional.}

For~$Q\in H^1(\Omega_\eps, \, \mcS_0)$, we consider the free energy functional
\begin{equation} \label{LdG}
 \mcF_\eps[Q] := \int_{\Omega_\eps} \left(f_e(\nabla Q) + f_b(Q)\right)\d x +
 \eps^{3-2\alpha} \sum_{j=1}^J \int_{\partial\mcP_\eps^j} f_s^j(Q, \, \nu) \, \d\sigma.
\end{equation}
The surface anchoring energy densities depend on~$j$, as
colloids that belong to different populations may have different surface properties.
For the rest, our assumptions for the elastic energy density~$f_e$,
bulk energy density~$f_b$, and surface energy densities~$f_s^j$
are the same as in~\cite{CanevariZarnescu}.
We say that a function $f\colon\mcS_0\otimes\R^3\to\R$ 
is strongly convex if there exists~$\theta>0$ such that 
the function $\mcS_0\otimes\R^3\ni D \mapsto f(D) - \theta|D|^2$
is convex. 

\begin{enumerate}[label=(H\textsubscript{\arabic*}), ref=H\textsubscript{\arabic*}, resume]
 \item \label{hp:fe}
 $f_e\colon\mcS_0\otimes\R^3\to[0, \, +\infty)$ 
 is differentiable and strongly convex. Moreover, 
 there exists a constant~$\lambda_e>0$ such that 
 \[
  \lambda_e^{-1}|D|^2 \leq f_e(D) \leq \lambda_e|D|^2, \qquad
  |(\nabla f_e)(D)| \leq \lambda_e\left(|D| + 1\right)
 \]
 for any~$D\in\mcS_0\otimes\R^3$.
 
 \item \label{hp:fb} $f_b\colon\mcS_0\to\R$ is continuous, non-negative
 and there exists~$\lambda_b>0$ such that
 $0 \leq f_b(Q)\leq \lambda_b(|Q|^6 + 1)$ for any~$Q\in\mcS_0$.
 
 \item \label{hp:fs} \label{hp:last} For any~$j\in\{1, \, \ldots, \, J\}$,
 the function $f^j_s\colon\mcS_0\times\mathbb{S}^2\to\R$ is 
 locally Lipschitz-continuous. Moreover, there exists a 
 constant~$\lambda_s>0$ such that
 \[
  |f^j_s(Q_1, \, \nu_1) - f_s^j(Q_2, \, \nu_2)|\leq
  \lambda_s\left(|Q_1|^3 + |Q_2|^3 + 1\right)
  \left(|Q_1 - Q_2| + |\nu_1 - \nu_2|\right)
 \]
 for any~$j\in\{1, \, \ldots, \, J\}$ and 
 any~$(Q_1, \, \nu_1)$, $(Q_2, \, \nu_2)$ in~$\mcS_0\times\mathbb{S}^2$.
\end{enumerate}

A physically relevant example of elastic energy density~$f_e$
that satisfies~\eqref{hp:fe} is given by
\begin{equation} \label{LdG-fe}
 f_e^{LdG}(\nabla Q) := L_1 \, \partial_k Q_{ij} \, \partial_k Q_{ij}
 + L_2 \, \partial_j Q_{ij} \, \partial_k Q_{ik}
 + L_3 \, \partial_j Q_{ik} \, \partial_k Q_{ij}
\end{equation}
(Einstein's summation convention is assumed), so long as the
coefficients~$L_1$, $L_2$, $L_3$ satisfy
\begin{equation} \label{Longa}
 L_1 > 0, \qquad - L_1 < L_3 < 2L_1,
 \qquad -\frac{3}{5} L_1 - \frac{1}{10} L_3 < L_2
\end{equation}
(see e.g.~\cite{deGennes, Longa}). The assumption~\eqref{hp:fb} is satisfied
by the quartic Landau-de Gennes bulk potential, given by
\begin{equation*} 
 f_b^{LdG}(Q) := a\,\tr(Q^2) - b\,\tr(Q^3) +
 c\left(\tr(Q^2)\right)^2 + \kappa(a, \, b, \, c)
\end{equation*}
where~$a\in\R$, $b>0$, $c>0$ are coefficients depending 
on the material and the temperature and~$\kappa(a, \, b,\ , c)\in\R$ 
is a constant, chosen in such a way that $\inf f^{LdG}_b = 0$.
An example of surface energy density that satisfies~\eqref{hp:fs} is
the Rapini-Papoular type energy density:
\begin{equation*} \label{Rapini}
 \begin{split}
  f_s(Q, \, \nu) &:= W \, \tr(Q - Q_\nu)^2
  \qquad \textrm{with } Q_\nu := \nu\otimes\nu - \frac{\Id}{3}
 \end{split}
\end{equation*}
and~$W$ a (typically positive) parameter. However, \eqref{hp:fs}
allows for much more general surface energy densities, 
which may not be positive and may have up to quartic growth in~$Q$
(for examples, see e.g. \cite{Goossens, SluckinPon, genpot, Rey}
and the references therein). In addition to~\eqref{hp:fs}, 
physically relevant surface energy densities must satisfy
symmetry properties (frame-indifference, invariance with respect 
to the sign of~$\nu$) but these will play no r\^ole
in our analysis.

\paragraph*{The homogenised potential.}

For any~$j\in\{1, \, \ldots, \, J\}$, let us define
$f_{hom}^j\colon\mcS_0\times\overline{\Omega}\to\R$ as
\begin{equation} \label{f_hom_j}
 f_{hom}^j(Q, \, x) := \int_{\partial\mcP^j} 
 f_s^j(Q, \, R_*^j(x)\nu_{\mcP^j}) \, \d\sigma
 \qquad \textrm{for } (Q, \, x)\in\mcS_0\times\overline{\Omega},
\end{equation}
where~$\nu_{\mcP^j}$ denotes the 
\emph{inward}-pointing unit normal to~$\partial\mcP^j$,
and~$R_*^j\colon\overline{\Omega}\to\SO(3)$ is the map
given by~\eqref{hp:R}. 
Finally, let
\begin{equation} \label{f_hom}
 f_{hom}(Q, \, x) := \sum_{j=1}^J \xi^j(x) f_{hom}^j(Q, \, x)
 \qquad \textrm{for } (Q, \, x)\in\mcS_0\times\overline{\Omega},
\end{equation}
where~$\xi^j\in L^\infty(\Omega)$ is the function given by~\eqref{hp:conv}.
Our candidate homogenised functional is defined
for any $Q\in H^1(\Omega, \mcS_0)$ as
\begin{equation} \label{Eq:homLdG}
 \mcF_0[Q] 
 := \int_{\Omega} \left(f_e(\nabla Q) + f_b(Q) + f_{hom}(Q, \, x)\right)\d x.
\end{equation}

\paragraph*{The convergence result.}

The assumptions~\eqref{hp:first}--\eqref{hp:last} are \emph{not}
enough to guarantee that global minimisers of~$\mcF_\eps$ exist 
and actually, it may happen that~$\mcF_\eps$ is unbounded from 
below \cite[Lemma~3.6]{CanevariZarnescu}. Instead, our main result 
focus on the asymptotic behaviour of \emph{local} minimisers.
Given~$g\in H^{1/2}(\partial\Omega, \, \mcS_0)$, we let
$H^1_g(\Omega_\eps, \mcS_0)$ --- respectively, $H^1_g(\Omega, \, \mcS_0)$ ---
be the set of maps~$Q\in H^1(\Omega_\eps, \mcS_0)$
--- respectively, $Q\in H^1(\Omega, \mcS_0)$ --- that satisfy~$Q = g$
on~$\partial\Omega$, in the sense of traces. For each~$Q\in H^1_g(\Omega_\eps, \mcS_0)$,
we define the map~$E_\eps Q\in H^1_g(\Omega, \, \mcS_0)$
by $E_\eps Q := Q$ on~$\Omega_\eps$ and 
$E_\eps Q := Q_\eps^{i,j}$ on~$\mcP_\eps^{i,j}$, where~$Q_\eps^{i,j}$
is the unique solution of Laplace's problem
\begin{equation} \label{harmonic}
 \begin{cases}
  -\Delta Q_\eps^{i,j} = 0 & \textrm{in } \mcP_\eps^{i,j} \\
  Q_\eps^{i,j} = Q         & \textrm{on } \partial\mcP_\eps^i.
 \end{cases}
\end{equation}

\begin{theorem}  \label{th:loc-min}
 Suppose that the assumptions \eqref{hp:first}--\eqref{hp:last}
 are satisfied. Suppose, moreover, that $Q_0\in H^1_g(\Omega, \mcS_0)$
 is an isolated $H^1$-local minimiser for 
 $\mcF_0$ --- that is, there exists $\delta_0>0$ such that
 \begin{equation*} 
  \mcF_0[Q_0] < \mcF_0[Q]
 \end{equation*}
 for any~$Q\in H^1_g(\Omega, \mcS_0)$ such that $Q\neq Q_0$ and
 $\|Q - Q_0\|_{H^1(\Omega)}\leq\delta_0$. Then, for any~$\varepsilon$
 small enough, there exists an $H^1$-local minimiser $Q_\eps$ for~$\mcF_\eps$
 such that $E_\eps Q_\eps\to Q_0$ strongly in~$H^1(\Omega)$ as~$\eps\to 0$.
\end{theorem}

The proof of Theorem~\ref{th:loc-min} follows a variational approach, 
and is based on the following fact:

\begin{prop} \label{prop:liminf}
 Let $Q_\eps\in H^1_g(\Omega_\eps, \, \mcS_0)$
 be such that $E_\eps Q_\eps\rightharpoonup Q$ weakly in~$H^1(\Omega)$
 as~$\eps\to 0$. Then, there holds
 \begin{gather}
  \int_{\Omega} \left(f_e(\nabla Q) + f_b(Q)\right)\d x 
   \leq \liminf_{\eps\to 0} 
   \int_{\Omega_\eps}  \left(f_e(\nabla Q_\eps) + f_b(Q_\eps)\right)\d x 
   \label{liminf_volume} \\
  \int_\Omega f_{hom}(Q, \, x) \, \d x 
   = \lim_{\eps\to 0} 
   \eps^{3-2\alpha} \sum_{j=1}^J \int_{\partial\mcP_\eps^j} 
   f_s^j(Q_\eps, \, \nu) \, \d\sigma \label{lim_jump}.
 \end{gather}
\end{prop}

Proposition~\ref{prop:liminf} can be reformulated as a $\Gamma$-convergence result.
Indeed, from Proposition~\ref{prop:liminf} we immediately have 
$\mcF_0\leq\Gamma\textrm{-}\liminf_{\eps\to 0}\mcF_\eps$ 
(with respect to a suitable topology, induced by the operator~$E_\eps$).
A trivial recovery sequence suffices to obtain the opposite~$\Gamma$-lim~sup inequality,
thanks to~\eqref{lim_jump} and the fact that in the dilute limit, $\abs{\mcP_\eps}\to 0$.
Theorem~\ref{th:loc-min} follows from Proposition~\ref{prop:liminf} by
general properties of the $\Gamma$-convergence.

Throughout the paper,
we will write $A\lesssim B$ as a short-hand for~$A\leq C B$,
where~$C$ is a positive constant, depending only on the domain,
the boundary datum and the free energy functional~\eqref{LdG},
but not on~$\eps$.

\subsection{Preliminary results}
\label{subsec:prelim}
The main technical tool is the following 
trace inequality, which is adapted from \cite[Lem\-ma~4.1]{ferromagnetic}.

\begin{lemma}[{\cite[Lemma~3.1]{CanevariZarnescu}}] \label{lem:trace}
 Let~$\mcP\subseteq\R^3$ be a compact,
 convex set whose interior contains the origin.
 Then, there exists a constant $C = C(\mcP)>0$
 such that, for any~$a>0$, $b\geq 2a$ and 
 any~$u\in H^1(b\mcP\setminus a\mcP)$, there holds
 \[
  \int_{\partial(a\mcP)} \abs{u}^4 \d\sigma \leq C 
  \int_{b\mcP\setminus a\mcP} \left(\abs{\nabla u}^2 + \abs{u}^6\right)\d x
  + \frac{Ca^2}{b^3} \int_{b\mcP\setminus a\mcP} \abs{u}^4\d x.
 \]
\end{lemma}

Given an inclusion~$\mcP_\eps^{i,j} = x_\eps^{i,j} + \eps^\alpha R_\eps^{i,j}\mcP^j$,
we consider~$\widehat{\mcP}_\eps^{i,j} := x_\eps^{i,j} + \mu\eps R_\eps^{i,j}\mcP^j$, 
where $\mu>0$ is a small (but fixed) parameter. By taking $\mu$ small enough,
we can make sure that the $\widehat{\mcP}_\eps^{i,j}$'s are pairwise disjoint.
Then, by applying Lemma~\ref{lem:trace} component-wise on
$\widehat{\mcP}_\eps^{i,j}\setminus\mcP_\eps^{i,j}$ and summing 
the corresponding inequalities over~$i$ and~$j$, we deduce

\begin{lemma} \label{lem:trace_eps}
 For any~$Q\in H^1(\Omega_\eps, \mcS_0)$, there holds
 \[
  \eps^{3-2\alpha} \int_{\partial\mcP_\eps} \abs{Q}^4 \d\sigma \lesssim 
    \eps^{3 - 2\alpha} \int_{\Omega_\eps} \left(\abs{\nabla Q}^2 + \abs{Q}^6\right)\d x
    + \int_{\Omega_\eps} \abs{Q}^4\d x.
 \]
\end{lemma}

Another tool is the harmonic extension operator,
$E_\eps\colon H^1(\Omega_\eps, \, \mcS_0)\to H^1(\Omega, \, \mcS_0)$,
defined by~\eqref{harmonic}.

\begin{lemma} \label{lem:extension}
 The operator~$E_\eps\colon 
 H^1(\Omega_\eps, \, \mcS_0)\to H^1(\Omega, \, \mcS_0)$
 satisfies the following properties.
 \begin{enumerate}[label=(\roman*)]
  \item There exists a constant $C>0$ such that 
   $\|\nabla (E_\eps Q)\|_{L^2(\Omega)} \leq C \|\nabla Q\|_{L^2(\Omega_\eps)}$
   for any~$\eps>0$ and any $Q\in H^1(\Omega_\eps, \mcS_0)$.
  \item If the maps $Q_\eps\in H^1(\Omega, \, \mcS)$
  converge $H^1(\Omega)$-strongly to $Q$ as~$\eps\to 0$,
  then $E_\eps(Q_{\eps|\Omega_\eps})\to Q$ 
  strongly in~$H^1(\Omega)$ as~$\eps\to 0$, too.
 \end{enumerate}
\end{lemma}
\begin{proof}
 For any~$i$, $j$, consider the inclusion 
 $\mcP_\eps^{i,j} = x_\eps^{i,j} + \eps^\alpha R_\eps^{i,j}\mcP^j$
 and let~$\mcR_{\eps}^{i,j} := x_\eps^{i,j} + 2\eps^\alpha R_\eps^{i,j}\mcP^j$.
 Let~$\mcR_{\eps} := \cup_{i,j} \mcR_{\eps}^{i,j}$.
 The properties of Laplace equation, combined with a
 scaling argument (see, e.g., \cite[Lemma~3.4]{CanevariZarnescu}),
 imply that
 \begin{equation} \label{harmonic1}
  \|\nabla (E_\eps Q)\|_{L^2(\mcP_\eps)} \lesssim 
  \|\nabla Q\|_{L^2(\mcR_\eps\setminus\mcP_\eps)}.
 \end{equation}
 Statement~(i) then follows immediately. To prove Statement~(ii),
 take a sequence~$Q_\eps\in H^1(\Omega, \,\mcS_0)$ that converges strongly to~$Q$ 
 as~$\eps\to 0$. Then,
 \[
  \begin{split}
   \norm{\nabla Q_\eps - \nabla (E_\eps(Q_{\eps|\Omega_\eps}))}_{L^2(\Omega)}
   &\leq \norm{\nabla Q_\eps}_{L^2(\mcP_\eps)} 
    + \norm{\nabla (E_\eps(Q_{\eps|\Omega_\eps}))}_{L^2(\mcP_\eps)} \\
   &\stackrel{\eqref{harmonic1}}{\lesssim}
    \norm{\nabla Q_\eps}_{L^2(\mcR_\eps)} 
    \lesssim \norm{\nabla Q}_{L^2(\mcR_\eps)}
    + \norm{\nabla Q - \nabla Q_\eps}_{L^2(\Omega)} \! .
  \end{split}
 \]
 Both terms in the right-hand side converge to~$0$
 as~$\eps\to 0$, because $|\mcR_\eps|\lesssim\eps^{3\alpha-3}\to 0$,
 and Statement~(ii) follows.
\end{proof}

\subsection{Proof of Theorem~\ref{th:loc-min}}
\label{subsec:proofhom}
The proof of Theorem~\ref{th:loc-min} is largely similar 
to that of \cite[Theorem~1.1]{CanevariZarnescu}. 
We reproduce here only some steps of the proof,
either because there require a modification or because
they will be useful in Section~\ref{sect:gamma}.

\paragraph*{Remarks on the lower semi-continuity of~$\mcF_\eps$.}

Even before we address the asymptotic analysis as~$\eps\to 0$,
we should make sure that, for fixed~$\eps>0$, the functional~$\mcF_\eps$
is sequentially lower semi-continuous with respect to 
the weak topology on~$H^1(\Omega_\eps, \, \mcS_0)$. If the surface energy
density~$f_s$ is bounded from below, then the surface integral is
lower semi-continuous by
Fatou lemma. If~$f_s$ has subcritical growth,
that is $\abs{f_s(Q)}\lesssim \abs{Q}^p + 1$ for some~$p<4$,
then the lower semi-continuity of 
the surface integral follows from the compact Sobolev embedding
$H^{1/2}(\partial\Omega_\eps, \, \mcS_0)\hookrightarrow L^p(\partial\Omega_\eps, \, \mcS_0)$
and Lebesgue's dominated convergence theorem. 
However, our assumption~\eqref{hp:fs} allows for surface energy densities
that have quartic growth and are unbounded from below, e.g. 
$f_s(Q) := -\abs{Q}^4$. In this case, the surface integral alone
may \emph{not} be sequentially weakly lower-semi 
continuous \cite[Lemma~3.10]{CanevariZarnescu}.
However, lower semi-continuity may be restored at 
least on \emph{bounded} subsets of~$H^1(\Omega_\eps, \, \mcS_0)$,
when~$\eps$ is small:

\begin{prop}[{\cite[Proposition~3.9]{CanevariZarnescu}}] \label{prop:lsc}
 Suppose that the assumptions~\eqref{hp:first}--\eqref{hp:last} are satisfied.
 For any~$M>0$ there exists~$\eps_0(M)>0$ such that following statement holds:
 if~$0 < \eps\leq \eps_0(M)$,
 $Q_k\rightharpoonup Q$ weakly in~$H^1_g(\Omega_\eps, \, \mcS_0)$
 and if~$\|\nabla Q_k\|_{L^2(\Omega_\eps)}\leq M$ for any~$k$, then
 \[
  \mcF_\eps[Q] \leq \liminf_{k\to+\infty} \mcF_\eps[Q_k].
 \]
\end{prop}

The proof carries over from~\cite{CanevariZarnescu}, almost word by word,
using Lemma~\ref{lem:trace_eps}.
Essentially, the loss of lower semi-continuity
that may arise from the surface integral can be quantified,
with the help of Lemma~\ref{lem:trace_eps} and the bound
on~$\nabla Q_k$. However,
since the surface integral is multiplied by a \emph{small}
factor~$\eps^{3 - 2\alpha}$, this loss of lower semi-continuity
is compensated by the strong convexity of the elastic term~$f_e$,
for~$\eps$ sufficiently small.

\paragraph*{Pointwise convergence of the surface energy terms.}
For ease of notation, let us define
\begin{gather}
 J_\eps[Q] := \eps^{3-2\alpha} \sum_{j=1}^J
  \int_{\partial\mcP_\eps^j} f_s^j(Q, \, \nu)\, \d\sigma \label{Jeps} \\
 J_0[Q] := \int_\Omega f_{hom} (Q, \, x) \, \d x \qquad
 \textrm{for } Q\in H^1_g(\Omega, \mcS_0) \label{J0}
\end{gather}

We state some properties of the functions
$f_{hom}^j\colon\mcS_0\times\overline{\Omega}\to\R$,
$f_{hom}\colon\mcS_0\times\overline{\Omega}\to\R$,
defined by~\eqref{f_hom}, \eqref{f_hom_j} respectively.
\begin{lemma} \label{lemma:fhom}
 For any~$j\in\{1, \, \ldots, \, J\}$, the function~$f_{hom}^j$
 is locally Lispchitz-continuous, and there holds
 \begin{equation}
  \abs{f_{hom}^j(Q, \, x)} \lesssim \abs{Q}^4 + 1, \qquad 
  \abs{\nabla f_{hom}^j(Q, \, x)} \lesssim 
  \abs{Q}^3 + 1  \label{fhomj-Lipschitz} 
 \end{equation}
 for any~$(Q, \, x)\in\mcS_0\times\overline{\Omega}$.
 Moroever, the function $f_{hom}$ satisfies
 \begin{equation} \label{fhom-Lipschitz}
  \abs{f_{hom}(Q_1, \, x) - f_{hom}(Q_2, \, x)} \lesssim
  \left(\abs{Q_1}^3 + \abs{Q_2}^3 + 1\right) \abs{Q_1 - Q_2}
 \end{equation}
 for any~$Q_1$, $Q_2\in\mcS_0$ and any~$x\in\overline{\Omega}$.
\end{lemma}
\begin{proof}
 Using the definition~\eqref{f_hom_j} of~$f_{hom}^j$,
 and the assumption~\eqref{hp:fs}, we obtain
 \[
  \begin{split}
   &\abs{f_{hom}^j(Q_1, \, x_1) - f_{hom}^j(Q_2, \, x_2)} 
   \leq \int_{\partial\mcP^j} \abs{f_s^j(Q_1, \, R_*^j(x_1)\nu_{\mcP^j})
     - f_s^j(Q_2, \, R_*^j(x_2)\nu_{\mcP^j})} \d\sigma \\
   &\qquad\qquad \lesssim \int_{\partial\mcP^j} \left(\abs{Q_1}^3 + \abs{Q_2}^3 + 1\right)
     \left(\abs{Q_1 - Q_2} + \abs{\left(R_*^j(x_1) - R_*^j(x_2)\right)\nu_{\mcP^j}}\right)\d\sigma
  \end{split}
 \]
 Since~$R_*^j$ is Lipschitz-continuous by~\eqref{hp:R}, we deduce
 \begin{equation} \label{fhom1}
  \begin{split}
   &\abs{f_{hom}^j(Q_1, \, x_1) - f_{hom}^j(Q_2, \, x_2)} 
   \lesssim \left(\abs{Q_1}^3 + \abs{Q_2}^3 + 1\right)
     \left(\abs{Q_1 - Q_2} + \abs{x_1 - x_2}\right)
  \end{split}
 \end{equation}
 and~\eqref{fhom-Lipschitz} follows. We multiply the
 previous inequality by~$\xi^j$, where~$\xi^j$ is given by~\eqref{hp:conv},
 take~$x_1 = x_2 = x$, and sum over~$j$. Since~$\xi^j\in L^\infty(\Omega)$
 by Assumption~\eqref{hp:conv}, we obtain
 \begin{equation*}
  \begin{split}
   \abs{f_{hom}(Q_1, \, x) - f_{hom}(Q_2, \, x)}
   &\stackrel{\eqref{f_hom}}{\leq} \sum_{j=1}^J \norm{\xi^j}_{L^\infty(\Omega)}
     \abs{f_{hom}^j(Q_1, \, x) - f_{hom}^j(Q_2, \, x)} \\
   &\lesssim \left(\abs{Q_1}^3 + \abs{Q_2}^3 + 1\right)
     \abs{Q_1 - Q_2} \! .  \qedhere
  \end{split}
 \end{equation*}
\end{proof}

Let us introduce the auxiliary quantity
\begin{equation} \label{Jtilde_eps}
  \tilde{J}_\eps[Q] := \eps^{3-2\alpha} \sum_{j=1}^J \sum_{i=1}^{N_\eps^j}
  \int_{\partial\mcP_\eps^{i,j}} f_s^j(Q(x_\eps^{i,j}), \, \nu)\, \d\sigma.
\end{equation}
\begin{lemma} \label{lemma:Jeps}
 For any bounded, Lipschitz map
 $Q\colon\overline{\Omega}\to\mcS_0$, there holds
 \begin{equation} \label{Jtilde_eps_repr} 
  \tilde{J}_\eps[Q] = \sum_{j=1}^J
    \int_{\R^3} f_{hom}^j(Q(x), \, x) \, \d\mu_\eps^j(x)
 \end{equation}
 (where the measures~$\mu_\eps^j$ are defined by~\eqref{hp:conv}), and
 \begin{equation} \label{Jtilde_eps_conv}
  \abs{J_\eps[Q] - \tilde{J}_\eps[Q]} \lesssim 
   \eps^{\alpha} \left(\norm{Q}_{L^\infty(\Omega)}^3 + 1\right) 
   \norm{\nabla Q}_{L^\infty(\Omega)} \! .
 \end{equation}
\end{lemma}
\begin{proof}
 For any~$i$ and~$j$, consider the single
 inclusion~$\mcP_\eps^{i,j} := x_\eps^{i,j} 
 + \eps^\alpha R_\eps^{i,j}\mcP_\eps^j$. 
 Since $\nu(x) = R_\eps^{i,j} \, \nu_{\mcP^j}
 (\eps^{-\alpha}(R_\eps^{i,j})^{\mathsf{T}}(x-x_\eps^{i,j}))$
 for any~$x\in\partial\mcP_\eps^{i,j}$, by a change of variable we obtain
 \[
  \begin{split}
   \tilde{J}_\eps[Q] &= \eps^3 \sum_{j=1}^J \sum_{i=1}^{N_\eps^j}
   \int_{\partial\mcP^j} f_s^j(Q(x_\eps^{i,j}), \, R_\eps^{i,j}\nu_{\mcP})\, \d\sigma \\
   &\stackrel{\eqref{hp:R}}{=} 
   \eps^3 \sum_{j=1}^J \sum_{i=1}^{N_\eps^j} \int_{\partial\mcP^j} 
     f_s^j(Q(x_\eps^{i,j}), \, R_*^j(x_\eps^{i,j})\nu_{\mcP})\, \d\sigma \\
   &\stackrel{\eqref{f_hom}}{=} \eps^3 \sum_{j=1}^J \sum_{i=1}^{N_\eps^j} 
     f_{hom}^j(Q(x_\eps^{i,j}), \, x_\eps^{i,j}).
  \end{split}
 \]
 Now, \eqref{Jtilde_eps_repr} follows from the
 definition of~$\mu_\eps^j$, \eqref{hp:conv}.
 On the other hand, by decomposing the integral on~$\partial\mcP^j$
 as a sum of integrals over the boundary of each inclusion, we obtain
 \[
  \begin{split}
   \abs{J_\eps[Q] - \tilde{J}_\eps[Q]}
   &\lesssim \eps^{3-2\alpha} 
   \sum_{j=1}^J \sum_{i=1}^{N_\eps^j} \int_{\partial\mcP_\eps^{i,j}}
   \abs{f_s^j(Q(x), \, \nu) - f_s^j(Q(x_\eps^{i,j}), \, \nu)} \d \sigma(x) \\
   &\stackrel{\eqref{hp:fs}}{\lesssim} \eps^{3-2\alpha} 
   \sum_{i=1}^{N_\eps^j} \sum_{j=1}^J \int_{\partial\mcP_\eps^{i,j}}
   \left(\abs{Q(x)}^3 + \abs{Q(x_\eps^{i,j})}^3 + 1\right)
   \abs{Q(x) - Q(x_\eps^{i,j})} \d \sigma(x) \\
  \end{split}
 \]
 Since~$Q$ is assumed to be Lipschitz continuous and the diameter
 of~$\mcP_\eps^{i,j}$ is~$\lesssim\eps^\alpha$, we have 
 $|Q(x) - Q(x_\eps^{i,j})|\lesssim \eps^\alpha 
 \norm{\nabla Q}_{L^\infty(\Omega)}$. This implies
 \[
  \abs{J_\eps[Q] - \tilde{J}_\eps[Q]} \lesssim
   \eps^{3-\alpha} \sigma(\partial\mcP_\eps)
   \left(\norm{Q}_{L^\infty(\R^3)}^3 + 1\right) 
   \norm{\nabla Q}_{L^\infty(\Omega)} \! .
 \]
 Finally, we note that $\sigma(\partial\mcP_\eps)\lesssim\eps^{2\alpha-3}$,
 because there are $\mathrm{O}(\eps^{-3})$ inclusions
 (as a consequence of~\eqref{hp:Omega_eps}) and the surface area of 
 each inclusion is~$\mathrm{O}(\eps^{2\alpha})$. 
 Thus, the lemma follows.
\end{proof}

Since~$\mu_\eps^j\rightharpoonup^* \xi^j\d x$ by 
Assumption~\eqref{hp:conv}, as an immediate consequence
of Lemma~\ref{lemma:Jeps} and~\eqref{f_hom}, \eqref{J0} we obtain

\begin{prop} \label{prop:formal}
 For any \emph{bounded, Lipschitz} map
 $Q\colon\overline{\Omega}\to\mcS_0$, there holds
 $J_\eps[Q]\to J_0[Q]$ as~$\eps\to 0$.
\end{prop}

Once Proposition~\ref{prop:formal} is proved, the rest of the proof of 
Proposition~\ref{prop:liminf} and
Theorem~\ref{th:loc-min} follows exactly as in~\cite{CanevariZarnescu}.

\section{Linear terms in the homogenised bulk potential}
\label{sec:linear}

\begin{proposition}\label{prop:wedgesbasis}
There exist (possibly disconnected) 
shapes~$\mathfrak{P}_k\subset\mathbb{R}^3$, $k\in\{1,2,3,4,5,6\}$
such that taking as surface energy the Rapini-Papoular surface 
energy $f_s(Q,\nu)=\tr(Q-Q_\nu)^2$ with $Q_\nu=\nu\otimes \nu-\frac{1}{3}\Id$
where $\nu$ is the exterior unit-normal, we have:
\be \label{form:fhompotato}
f_{hom}^{\mathfrak{P}_k}(Q)= \left(\frac{2}{3}+\tr(Q^2)\right)
\sigma(\partial\mathfrak{P}_k)-2\tr(QM_k)
\ee where 
$$
M_k=\left(\frac{\pi}{3}+\frac{\pi}{2}\right)\Id-\frac{\pi}{2}e_k\otimes e_k,
\qquad k\in\{1,2,3\}
$$

$$
M_4=\left(\frac{\pi}{3}+\frac{\pi}{2}\right)\Id-\frac{\pi}{2}e_3\otimes e_3+\frac{2}{3}(e_1\otimes e_2+e_2\otimes e_1), 
$$

$$
M_5=\left(\frac{\pi}{3}+\frac{\pi}{2}\right)\Id-\frac{\pi}{2}e_2\otimes e_2+\frac{2}{3}(e_1\otimes e_3+e_3\otimes e_1), 
$$

$$
M_6=\left(\frac{\pi}{3}+\frac{\pi}{2}\right)\Id-\frac{\pi}{2}e_1\otimes e_1+\frac{2}{3}(e_2\otimes e_3+e_3\otimes e_2), 
$$
with $M_k,\, k\in\{1,2,3,4,5,6\}$ a basis in the linear space of $3\times 3$ symmetric matrices.
\end{proposition}
\begin{proof}
By formula \eqref{f_hom_j} we have:
\be 
f_{hom}^{\mathfrak{P}_k}(Q) =\int_{\partial\mathfrak{P}_k}
\left(\tr(Q^2)-2\tr(QQ_\nu)+\tr(Q_\nu)^2\right) \d\sigma
\ee
hence we readily get \eqref{form:fhompotato} with
$M_k=\int_{\partial\mathfrak{P}_k} \nu(x)\otimes \nu(x)\, \d\sigma(x)$.

Let us take for $1\le i,j\le 3$ with $i\not=j$ the {\it`potato wedges'} domains 
\begin{equation} \label{Omega+}
 \Omega_{ij}^+:= \{x=(x_1,x_2,x_3)\in\R^3\colon |x|\le 1, \  x_i\geq 0, \ x_j\ge 0\}
\end{equation}
as candidates for {\it `parts of'} our shapes $\mathfrak{P}_k$'s
(see Figure~\ref{fig:PotatoWedges}). 
\begin{figure}[t]
		\centering
		\includegraphics[height=.25\textheight]{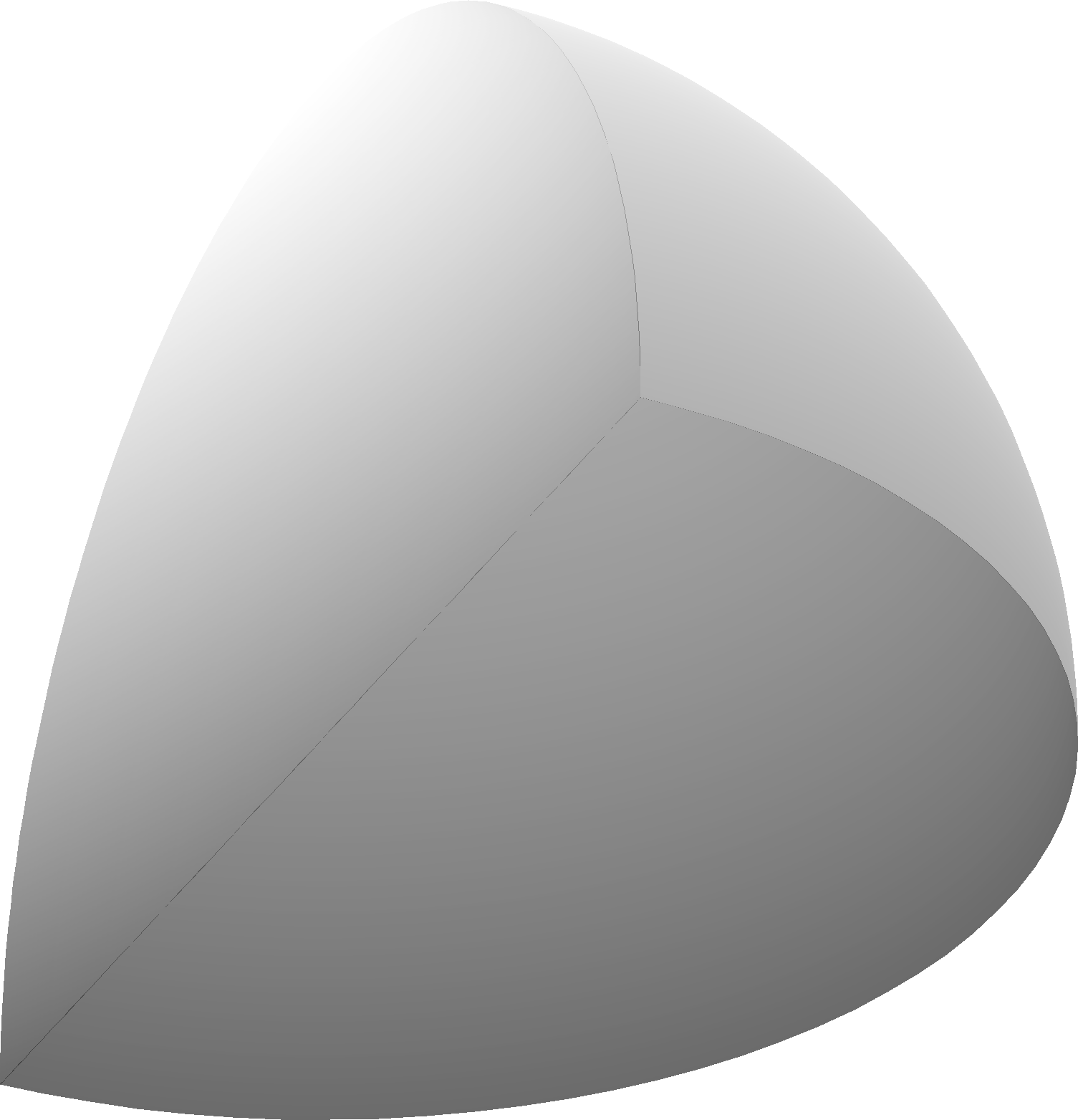}
	\caption{The `potato wedge' domain~$\Omega_{23}^+$.}
	\label{fig:PotatoWedges}
\end{figure}
We calculate the term
$$
\int_{\partial\Omega_{12}^+}\nu\otimes \nu\,\d\sigma
=\underbrace{\int_{\Omega_{12}^+\cap\{x_1=0\}}\nu\otimes \nu\,\d\sigma}_{:=\mathcal{I}_1}+\underbrace{\int_{\Omega_{12}^+\cap\{x_2=0\}}\nu\otimes \nu\,\d\sigma}_{:=\mathcal{I}_2}+\underbrace{\int_{\partial\Omega_{12}^+\cap\{x_1\cdot x_2>0\}}\nu\otimes \nu\,\d\sigma}_{:=\mathcal{I}_3}
$$
Then:
\be
\mathcal{I}_1=e_1\otimes e_1\int_{\{x_2^2+x_3^2\le 1,\, x_1=0\}} \d\sigma=\frac{\pi}{2}e_1\otimes e_1
\ee
where $e_1:=(1,0,0)$. Similarly we get $\mathcal{I}_2=\frac{\pi}{2}e_2\otimes e_2$ with $e_2:=(0,1,0)$.
Finally:
\be
(\mathcal{I}_3)_{ij}=\int_{\partial\Omega_{12}^+\cap\{x_1\cdot x_2>0\}} x_ix_j\,\d\sigma(x), 
\qquad \forall \, i, j\in\{1,2,3\}.
\ee
Because $x_1x_3$, respectively $x_2x_3$ are odd functions in the variable $x_3$ along the domain $\Omega_{12}^+\cap\{x_1\cdot x_2>0\}$ we have that $(\mathcal{I}_3)_{13}=(\mathcal{I}_3)_{31}=(\mathcal{I}_3)_{23}=(\mathcal{I}_3)_{32}=0$.
Furthermore:
\bea
(\mathcal{I}_3)_{12}=(\mathcal{I}_3)_{21}
&=\int_{\partial\Omega_{12}^+\cap\{x_1\cdot x_2>0\}} x_1x_2\,\d\sigma_x
=\int_0^\pi\left(\int_0^{\frac{\pi}{2}} \sin^3\theta\cos\varphi\sin\varphi\,\,\d\varphi\right)\d\theta\non\\
&=\int_0^\pi\sin^3\theta\,\d\theta\cdot\int_0^{\frac{\pi}{2}}\cos\varphi\sin\varphi\,\,\d\varphi
=\frac{4}{3}\cdot\frac{1}{2}=\frac{2}{3}
\eea
Similarly we get: $(\mathcal{I}_3)_{11}=(\mathcal{I}_3)_{22}=(\mathcal{I}_3)_{33}=\frac{\pi}{3}$. Summarizing the last calculations, we get:
\be
\int_{\partial\Omega_{12}^+}\nu\otimes \nu\,\d\sigma
=\left(\begin{array}{lll}\frac{\pi}{3}+\frac{\pi}{2} & \frac 23 & 0\\
\frac 23 & \frac{\pi}{3}+\frac{\pi}{2} & 0\\ 0 & 0 & \frac{\pi}{3}\end{array}\right)
\ee 
Analogous calculations provide
\be
\int_{\partial\Omega_{13}^+} \nu\otimes \nu\,\d\sigma
=\left(\begin{array}{lll}\frac{\pi}{3}+\frac{\pi}{2} & 0 & \frac{2}{3} \\0 &\frac{\pi}{3} & 0\\ \frac 23 & 0 &\frac{\pi}{3}+\frac{\pi}{2}  \end{array}\right)
\ee
\be
\int_{\partial\Omega_{23}^+}\nu\otimes \nu\,\d\sigma
=\left(\begin{array}{lll}\frac{\pi}{3} & 0 & 0\\
0 & \frac{\pi}{3}+\frac{\pi}{2} & \frac 23 \\ 0 & \frac 23 & \frac{\pi}{3}+\frac{\pi}{2}\end{array}\right)
\ee
Similarly we define, for $1\le i<j\le 3$,
\begin{equation} \label{Omega-}
 \Omega_{ij}^- := \{x=(x_1,x_2,x_3)\in\R^3
 \colon |x|\le 1, \ x_i\le 0, \ x_j\ge 0\}
\end{equation}
and we have:
\be
\int_{\partial\Omega_{12}^-}\nu\otimes\nu\, \d\sigma
=\left(\begin{array}{lll}\frac{\pi}{3}+\frac{\pi}{2} & -\frac 23 & 0\\
-\frac 23 & \frac{\pi}{3}+\frac{\pi}{2} & 0\\ 0 & 0 & \frac{\pi}{3}\end{array}\right)
\ee
respectively
\be
\int_{\partial\Omega_{13}^-} \nu\otimes\nu\, \d\sigma
=\left(\begin{array}{lll}\frac{\pi}{3}+\frac{\pi}{2} & 0 & -\frac{2}{3} \\0 &\frac{\pi}{3} & 0\\ -\frac 23 & 0 &\frac{\pi}{3}+\frac{\pi}{2}  \end{array}\right)
\ee
\be
\int_{\partial\Omega_{23}^-} \nu\otimes\nu \, \d\sigma
=\left(\begin{array}{lll}\frac{\pi}{3} & 0 & 0\\
0 & \frac{\pi}{3}+\frac{\pi}{2} & -\frac 23 \\ 0 & -\frac 23 & \frac{\pi}{3}+\frac{\pi}{2}\end{array}\right)
\ee
We take then:
\begin{gather*}
\mathfrak{P}_1:=\Omega_{23}^+\cup\left(\Omega_{23}^- - (0, \, 1, \, 0)\right), \quad 
\mathfrak{P}_2:=\Omega_{13}^+\cup\left(\Omega_{13}^- - (1, \, 0, \, 0)\right), \\
\mathfrak{P}_3:=\Omega_{12}^+\cup\left(\Omega_{12}^- - (1, \, 0, \, 0)\right)
\end{gather*}
and, respectively
\[
\mathfrak{P}_4:=\Omega_{12}^+, \qquad
\mathfrak{P}_5:=\Omega_{13}^+, \qquad
\mathfrak{P}_6:=\Omega_{23}^+.
\qedhere
\]
\end{proof}

\bigskip
\begin{proposition}\label{prop:homlinear}
Let $P$ be a $3\times 3$ symmetric matrix, not necessarily traceless, and $W\in\R$.
There exists a  family of $J_P\in\N$ shapes~$\mcP^j$ and corresponding surface energy 
strengths $i_j,j\in\{1,\dots, J_P\}$ such that,
taking for each shape the Rapini-Papoular surface energy
with corresponding intensity~$i_j$, i.e.
\begin{equation} \label{choice-of-fs}
 f_{s}^j(Q, \,\nu)=W\,i_j\,\tr(Q-Q_\nu)^2
\end{equation}
with $Q_\nu := \nu\otimes \nu-\frac{1}{3}\Id$ 
and~$\nu$ is the exterior unit-normal,
the corresponding homogenised potential is:
\be \label{form:homlin-bis}
 f_{hom}^P(Q) = - W\alpha_P \left(\frac{1}{3}+\frac{1}{2}\tr(Q^2)\right) + W\tr(QP)
\ee 
with $\alpha_P\in\R$ explicitly computable in terms of
the shapes volumes and the surface energy strengths.
\end{proposition}

\begin{proof}
We take $M_k$, $k\in\{1,\dots,6\}$, as provided in Proposition~\ref{prop:wedgesbasis},
to be  a linear basis in the spaces of $3\times 3$ symmetric matrices.
Then there exists $a_k:=P\cdot M_k$, $k\in\{1,\dots,6\}$ such that 
$$
P=\sum_{k=1}^6 a_kM_k
$$
Let~$\mcP^1, \, \ldots, \, \mcP^{J_P}$ be the connected components 
of~$\mathfrak{P}^1, \, \ldots, \, \mathfrak{P}^k$.
Each~$\mcP^j$ is a compact, convex set of the 
form~\eqref{Omega+} or~\eqref{Omega-} (see Figure~\ref{fig:PotatoWedges}). 
For~$j\in\{1, \, \ldots, \, J_P\}$, we 
define the corresponding intensities as $i_j=-\frac{1}{2}a_k$
where~$k = k(j)$ is such that~$\mcP^j\subseteq\mathfrak{P}^k$.
Then, noting that the homogenised potentials 
corresponding to each species will add together to provide the
homogenised porential corresponding to all the species, we get:
\be\label{form:homlin}
f_{hom}^P(Q) = -\frac{1}{2}\sum_{k=1}^7 a_k f_{hom}^{\mathfrak{P}^k}(Q)
=-W\left(\frac{1}{3}+\frac{1}{2}\tr(Q^2)\right)
\sum_{k=1}^6i_k\sigma(\partial\mathfrak{P}_k) + W\tr(QP)
\ee hence we obtain the claimed~\eqref{form:homlin-bis} with
$\alpha_P:=\sum_{k=1}^6i_k\sigma(\partial\mathfrak{P}_k)$.
\end{proof}

\begin{remark}\label{rmk:bulkconst}
We can, without loss of generality, drop the constant term in a bulk potential, since adding a constant to an energy functional does not change the minimiser. In particular in $f_{hom}^P$ we can ignore the term $-\frac{W}{3}\alpha_P$.
\end{remark}

We wish now to choose the surface energy densities~$f_s^j$ 
of Rapini-Papoular type and the shapes of the colloidal particles, 
in such a way that given the symmetric $3\times 3$ matrix~$P$ and~$W\in\R$,
local minimisers of the Landau-de Gennes functional
\begin{equation} \label{Feps-Ldg+}
 \begin{split}
 \mcF_\eps[Q] &= \int_{\Omega_\eps} \left(f_e^{LdG}(\nabla Q) 
 + a\,\tr(Q^2) - b\, \tr(Q^3) + c\left(\tr(Q^2)\right)^2 \right)\d x \\
 &\qquad\qquad\qquad + \eps^{3-2\alpha} 
 \sum_{j=1}^{J} \int_{\partial\mcP_\eps^j} f_s^j(Q, \, \nu) \, \d\sigma
 \end{split}
\end{equation}
(with~$f_e^{LdG}$ given by~\eqref{LdG-fe})
converge to local minimisers of the homogenised functional
\begin{equation} \label{Fzero-LdG+}
 \mcF_0[Q] = \int_\Omega \left(f_e^{LdG}(\nabla Q) 
 + a^\prime \, \tr(Q^2) - b \, \tr(Q^3) +
 c \left(\tr(Q^2)\right)^2 +W\tr(PQ) \right)\d x.
\end{equation}
We will assume that~$1 < \alpha < 3/2$
and the centres of the inclusions, $x^{i,j}_\eps$, satisfy~\eqref{hp:Omega_eps}
and that they are uniformly distributed, i.e. they satify~\eqref{hp:conv} with~$\xi^j = 1$.
We also assume that all inclusions of the same family are parallel to each other, 
that is, we take~$R^{i,j}_\eps= \mathrm{Id}$ for any~$i$, $j$, $\eps$
(in particular, \eqref{hp:R} is satisfied with~$R_*^j = \mathrm{Id}$).

\begin{remark} One could also choose colloidal particles and corresponding surface energies that modify the $b$ and $c$ coefficients, but for this it would not suffice to use Rapini-Papoular type of surface energies (see for instance Section~$2.2$ in~\cite{CanevariZarnescu}).
\end{remark}

\begin{corollary} \label{th:conv-LdGlinear}
 Let~$(a, \, b, \, c)\in\R^3$  with~$c>0$. Let~$a^\prime\in\R$, 
 $W>0$, and let~$P$ be a symmetric, $3\times 3$ matrix.
Suppose that the inequalities~\eqref{Longa}
 are satisfied.
 Then, there exists a family of shapes~$\mcP^j$ 
 and a corresponding surface energy~$f_s^j$ for each of them,
 such that for any isolated local minimiser~$Q_0$ of the functional~$\mcF_0$
 defined by~\eqref{Fzero-LdG+}, and for~$\eps>0$ small enough,
 there exists a local minimiser~$Q_\eps$ of the functional~$\mcF_\eps$,
 defined by~\eqref{Feps-Ldg+}, such that $E_\eps Q_\eps\to Q_0$
 strongly in~$H^1(\Omega, \, \mcS_0)$.
\end{corollary}
\begin{proof}
 This statement is a particular case of our main result,
 Theorem~\ref{th:loc-min}. 
 If~\eqref{Longa} holds and~$c>0$, $c^\prime>0$, then the 
 conditions~\eqref{hp:fe}--\eqref{hp:fb} are satisfied. 
 
 We take $J_P$ species $\mcP^j,\,j\in\{1,\dots,J_P\}$
 and surface energies given by~\eqref{choice-of-fs}, as in 
 Proposition~\ref{prop:wedgesbasis}. Each $\mcP_j$
 is a compact, convex set of the form~\eqref{Omega+} or~\eqref{Omega-},
 so~\eqref{hp:P} is satisfied (up to translations) 
 and~\eqref{hp:fs} is satisfied too.
 The homogenised potential corresponding to these is:
 \be 
  f_{hom}^P(Q) = - W \alpha_P\left(\frac{1}{3}+\frac{1}{2}\tr(Q^2)\right)+W\tr(QP)
 \ee
 where $\alpha_P:=\sum_{k=1}^6\frac{P\cdot M_k}{2}\sigma(\partial\mathfrak{P}_k)$
 (and the $M_k, k\in\{1,\dots,6\}$ are those from Proposition~\ref{prop:wedgesbasis}).
 We further take one more species, of spherical colloids~$\mcP^{J_P+1} := B_1$,
 and define the surface energy density
 \begin{equation} \label{fs-LdG}
  \begin{split}
   f_s^{J_P+1}(Q, \, \nu) &:= \frac{1}{4\pi}\left(a^\prime +\frac{W\alpha_P}{2}- a\right) (\nu\cdot Q^2\nu).
  \end{split}
 \end{equation}
 This produces (see also for instance Remark $2.9$ in \cite{CanevariZarnescu}) 
 a homogenised potential 
 $$
 f_{hom}^{sph}(Q) := \left(a^\prime +\frac{W\alpha_P}{2}- a\right)\tr(Q^2)
 $$
 Then the homogenised potential for all the $J_P+1$ species is
 \[
  \begin{split}
  f_{hom}(Q) &= f_{hom}^{sph}(Q) + f_{hom}^P(Q) \\
  &=(a^\prime - a)\,\tr(Q^2) +b\,\tr(Q^3)
  + c\,(\tr(Q^2))^2+W\tr(PQ)-\frac{W}{3}\alpha_P
  \end{split}
 \]
 and, since we can, without loss of generality, see Remark~\ref{rmk:bulkconst}
 drop the constant term $-\frac{W}{3}\alpha_P$, the corollary follows from Theorem~\ref{th:loc-min}.
\end{proof}

\section{The limit functional in the case of stronger anchoring strength}
\label{sect:gamma}

The purpose of this section is to study the asymptotic behaviour,
as~$\eps\to 0$, of minimisers of a functional with a different choice of the scaling
for the surface anchoring strength.
We consider the free energy functional:
\begin{equation} \label{Eq:LdG+}
 \mcF_{\eps,\gamma}[Q] 
 := \int_{\Omega_\eps} \left(f_e(\nabla Q) + f_b(Q)\right)\d x +
   \eps^{3-2\alpha-\gamma} \sum_{j=1}^{J} \int_{\partial\mcP_\eps^j} 
   f_s^j(Q, \, \nu) \, \d\sigma.
\end{equation} (where $\nu(x)$ denotes as usually 
the exterior normal at the point $x$ on the boundary), 
with $\alpha\in (1,\frac{3}{2})$ and
\begin{enumerate}[label=(K\textsubscript{\arabic*}), ref=K\textsubscript{\arabic*}]
 \item \label{hpgamma:gamma} \label{hpgamma:first} $0 < \gamma < 1/4$.
\end{enumerate}
Due to the extra factor~$\eps^{-\gamma}$ in front of the surface integral,
we \emph{cannot} apply Proposition~\ref{prop:lsc} to obtain the 
lower semi-continuity of~$\mcF_{\eps,\gamma}$ for fixed~$\eps$. 
Therefore, in contrast with the previous sections,
we assume boundedness from below on the surface term. 
\begin{enumerate}[label=(K\textsubscript{\arabic*}), ref=K\textsubscript{\arabic*}, resume]
 \item \label{hpgamma:fs} $f_s\geq 0$.
\end{enumerate}
\begin{remark} \label{remark:existence}
 Under the assumption~\eqref{hpgamma:fs}, the sequential weak
 lower semi-continuity of~$\mcF_\eps$ (for fixed~$\eps$)
 follows from the compact 
 embedding~$H^{1/2}(\partial\Omega_\eps)\hookrightarrow L^2(\partial\Omega_\eps)$ù
 and Fatou's lemma. Therefore, a routine application of the
 direct method of the Calculus of Variations shows that
 minimisers of~$\mcF_\eps$ exist, for any~$\eps>0$.
\end{remark}
As a consequence of~\eqref{hpgamma:fs} and of~\eqref{hp:conv},
the function~$f_{hom}$ is non-negative, too.
In fact, we will also assume that
\begin{enumerate}[label=(K\textsubscript{\arabic*}), ref=K\textsubscript{\arabic*}, resume]
 \item \label{hpgamma:fhom} $\inf\{f_{hom}(Q, \, x) \colon 
 Q\in\mcS_0\} = 0$ for any~$x\in\overline{\Omega}$.
\end{enumerate}
Recall that, for any~$j\in\{1, \, \ldots, \, J\}$,
the measures $\mu_\eps^j := \eps^{-3}\sum_i\delta_{x^{i,j}_\eps}$
are supposed to converge weakly$^*$ to a non-negative
function~$\xi^j\in L^\infty(\Omega)$.
We need to prescribe a rate of convergence. We express the rate of 
convergence in terms of the~$W^{-1,1}$-norm (that is, the dual
Lipschitz norm, also known as flat norm in some contexts):
\begin{equation} \label{flat}
 \begin{split}
  \mathbb{F}_\eps := \max_{j=1, \, 2, \, \ldots, \, J}
   &\sup \bigg\{ \int_{\Omega} \varphi\,\d\mu_\eps^j 
    - \int_{\Omega} \varphi\,\xi^j\,\d x\colon \\
   &\qquad\qquad \varphi\in W^{1,\infty}(\Omega), \quad 
    \norm{\nabla\varphi}_{L^\infty(\Omega)} + 
    \norm{\varphi}_{L^\infty(\Omega)} \leq 1 \bigg\} .
 \end{split}
\end{equation}
\begin{enumerate}[label=(K\textsubscript{\arabic*}), ref=K\textsubscript{\arabic*}, resume]
 \item \label{hpgamma:flat}
 There exists a constant~$\lambda_{\mathrm{flat}}>0$
 such that $\mathbb{F}_\eps \leq \lambda_{\mathrm{flat}}\eps$ for any~$\eps$.
\end{enumerate}

\begin{remark} \label{remark:flat}
 The assumption~\eqref{hpgamma:flat} is satisfied if the inclusions 
 are periodically distributed. Consider, for simplicity, the case~$J=1$,
 and suppose that the centres of the inclusions, $x_\eps^i$,
 are exactly the points~$y\in (\eps\mathbb{Z})^3$ such that
 $y + [-\eps/2, \, \eps/2]^3\subseteq\Omega$.
 Let~$\Omega_\eps := \cup_i (x_\eps^i + [-\eps/2, \, \eps/2]^3)$.
 Then, for any~$\varphi\in W^{1, \infty}(\Omega)$, we have
 \[
  \begin{split}
   \abs{\int_{\Omega} \varphi \,\d\mu_\eps - \int_{\Omega} \varphi \,\d x}
   &\leq \sum_{i = 1}^{N_\eps}\int_{x_\eps^{i} + [-\eps/2, \, \eps/2]^3} 
    \abs{\varphi - \varphi(x_\eps^{i})}
    + \int_{\Omega\setminus\Omega_\eps} \abs{\varphi} \\
   &\leq \frac{\sqrt{3}\,\eps}{2} \norm{\nabla\varphi}_{L^\infty(\Omega)} \abs{\Omega_\eps}
   + \norm{\varphi}_{L^\infty(\Omega)} \abs{\Omega\setminus\Omega_\eps}.
  \end{split}
 \]
 Moreover, $\abs{\Omega\setminus\Omega_\eps}\lesssim\eps$,
 because~$\Omega\setminus\Omega_\eps\subseteq
 \{y\in\Omega\colon \dist(y, \, \partial\Omega)\leq \sqrt{3}\,\eps\}$. 
 Therefore, \eqref{hpgamma:flat} holds.
\end{remark}

Finally, we assume some regularity on the boundary datum~$g\colon\partial\Omega\to\mcS_0$.
\begin{enumerate}[label=(K\textsubscript{\arabic*}), ref=K\textsubscript{\arabic*}, resume]
 \item \label{hpgamma:g} \label{hpgamma:last}
 $g$ is bounded and Lipschitz.
\end{enumerate}

\paragraph{$\Gamma$-convergence to a constrained problem.}

We can now define the candidate limit functional. Let
\begin{equation} \label{A}
 \mcA := \left\{Q\in H^1_g(\Omega, \, \mcS_0)\colon
 f_{hom}(x, \, Q(x)) = 0 \ \textrm{ for a.e. } x\in\Omega\right\} \! ,
\end{equation}
and $\widetilde{\mcF}\colon L^2(\Omega, \, \mcS_0)\to (-\infty, \, +\infty]$,
\[
 \widetilde{\mcF}(Q) := \begin{cases}
                        \displaystyle\int_{\Omega} 
                          \left(f_e(\nabla Q) + f_b(Q)\right)\d x 
                        & \textrm{if } Q\in\mcA \\
                          +\infty & \textrm{otherwise.}
                       \end{cases}
\]

\begin{theorem} \label{th:gamma}
  Suppose that the assumptions~\eqref{hp:first}--\eqref{hp:last},
  \eqref{hpgamma:first}--\eqref{hpgamma:last} are satisfied.
  Then, the following statements hold.
  \begin{enumerate}[label=(\roman*)]
   \item Given a family of maps~$Q_\eps\in H^1_g(\Omega_\eps, \, \mcS_0)$
   such that $\sup_\eps\mcF_{\eps,\gamma}(Q)<+\infty$,
   there exists a non-relabelled sequence and~$Q_0\in\mcA$
   such that $E_\eps Q_\eps \rightharpoonup Q_0$ weakly in~$H^1(\Omega)$,
   \[
    \widetilde{\mcF}(Q_0) \leq \liminf_{\eps\to 0} \mcF_{\eps,\gamma}(Q_\eps).
   \]
   \item For any $Q_0\in\mcA$, there exists a sequence of
   maps~$Q_\eps\in H^1_g(\Omega_\eps, \, \mcS_0)$
   such that $E_\eps Q_\eps\to Q_0$ strongly in~$H^1(\Omega)$ and
   \[
    \limsup_{\eps\to 0} \mcF_{\eps,\gamma}(Q_\eps) \leq \widetilde{\mcF}(Q_0).
   \]
  \end{enumerate}
\end{theorem}

\begin{remark} The theorem is only meaningful when~$\mcA$ is non-empty,
and it may happen that~$\mcA$ is empty even if~$f_{hom}(g(x), \, x) = 0$
for any~$x\in\partial\Omega$.
\end{remark}

\subsection{Proof of Theorem~\ref{th:gamma}}

Before we give the proof of Theorem~\ref{th:gamma},
we state some auxiliary results. We first recall some
properties of the convolution, which will be useful 
in constructing the recovery sequence.

\begin{lemma} \label{lemma:convolution}
 For any~$P\in H^1(\R^3, \, \mcS_0)$ and~$\sigma>0$, there exists 
 a smooth map~$P_\sigma\colon\R^3\to\mcS_0$ that satisfies 
 the following properties:
 \begin{gather}
  \norm{P_\sigma}_{L^\infty(\R^3)}
   \lesssim \sigma^{-1/2}\norm{P}_{L^6(\R^3)}, \quad
  \norm{\nabla P_\sigma}_{L^\infty(\R^3)}
   \lesssim \sigma^{-3/2} \norm{\nabla P}_{L^2(\R^3)} \label{conv_Lipschitz} \\
  \norm{P - P_\sigma}_{L^2(\R^3)} \lesssim 
   \sigma\norm{\nabla P}_{L^2(\R^3)} \label{conv_conv} \\
  \norm{\nabla P - \nabla P_\sigma}_{L^2(\R^3)} 
   \to 0 \qquad \textrm{as } \sigma\to 0. \label{conv_H1}
 \end{gather}
 Moreover, if~$U\subseteq U^\prime$ are Borel subsets of~$\R^3$
 such that $\dist(U, \, \R^3\setminus U^\prime)>\sigma$, then
 \begin{equation} \label{conv_local}
  \norm{P_\sigma}_{L^2(U)} \leq \norm{P}_{L^2(U^\prime)} \! .
 \end{equation}
\end{lemma}
\begin{proof} 
 Let us take a non-negative, even function~$\varphi\in C^\infty_{\mathrm{c}}(\R^3)$,
 supported in the unit ball~$B_1$, such that $\norm{\varphi}_{L^1(\R^3)} = 1$.
 Let $\varphi_\sigma(x) := \sigma^{-3}\varphi(x/\sigma)$.
 By a change of variable, we see that
 \begin{equation} \label{conv1}
  \norm{\varphi_\sigma}_{L^p(\R^3)} = \sigma^{3/p - 3}\norm{\varphi}_{L^p(\R^3)}
  \qquad \textrm{for any } p\in [1, \, +\infty).
 \end{equation}
 Let~$P_\sigma$ be defined as the convolution~$P_\sigma := P*\varphi_\sigma$.
 Then, by Young's inequality, we have
 \[
  \norm{\nabla P_\sigma}_{L^\infty(\R^3)} = 
  \norm{(\nabla P) * \varphi_\sigma}_{L^\infty(\R^3)}
  \leq \norm{\nabla P}_{L^2(\R^3)} \norm{\varphi_\sigma}_{L^2(\R^3)}
  \stackrel{\eqref{conv1}}{\lesssim} 
  \sigma^{-3/2}\norm{\nabla P}_{L^2(\R^3)} \! .
 \]
 The other inequality in~\eqref{conv_Lipschitz} is obtained in a similar way.
 The condition~\eqref{conv_H1} is a well-known property of convolutions.
 
 Let us prove~\eqref{conv_conv}. Let~$\psi$ be the Fourier 
 transform\footnote{We adopt the convention 
 $\widehat{\varphi}(\xi) = \int_{\R^3} \varphi(x)
 \exp(-2\pi i x\cdot\xi) \, \d x$.} of~$\varphi$. Then, 
 $\psi$ is smooth and rapidly decaying (that is, it belongs 
 to the Schwartz space~$\mcS(\R^3)$) and, in particular,
 it is Lipschitz continuous. Moreover,
 $\psi(0) = \int_{\R^3}\varphi(x)\,\d x = 1$. 
 By the properties of the Fourier transform, we have 
 $\widehat{\varphi_\sigma}(\xi) = \psi(\sigma\xi)$.
 By applying Plancherel theorem, we obtain
 \[
  \begin{split}
   \norm{P - P_\sigma}_{L^2(\R^3)}^2 
   &= \int_{\R^3} |\hat{P}(\xi)|^2(1 - \psi(\sigma\xi))^2 \, \d\xi \\
   &= \int_{\R^3} |\hat{P}(\xi)|^2(\psi(0) - \psi(\sigma\xi))^2 \, \d\xi \\
   &\leq \sigma^2 \norm{\nabla\psi}^2_{L^\infty(\R^3)} 
    \int_{\R^3} \abs{\xi}^2 |\hat{P}(\xi)|^2 \, \d\xi \\
   &= \frac{\sigma^2}{4\pi^2} \norm{\nabla\psi}^2_{L^\infty(\R^3)}
   \norm{\nabla P}_{L^2(\R^3)}^2 \! .
  \end{split}
 \]
 It only remains to prove~\eqref{conv_local}. Let~$\chi$
 be the indicator function of~$U^\prime$ (i.e.
 $\chi=1$ on~$U^\prime$ and~$\chi=0$ elsewhere).
 Observe that $P_\sigma = (\chi P)*\varphi_\sigma$ on~$U$, because~$\varphi_\sigma$
 is supported on the ball~$B_\sigma$ of radius~$\sigma$ and,
 by assumption, $\dist(U, \, \R^3\setminus U^\prime)>\sigma$. Then,
 Young inequality implies
 \[
  \norm{P}_{L^2(U)} \leq \norm{\chi P}_{L^2(\R^3)}
  \norm{\varphi_\sigma}_{L^1(\R^3)} = \norm{P}_{L^2(U^\prime)} \! .
  \qedhere
 \]
\end{proof}

\begin{lemma} \label{lemma:conv-bdd}
 Let~$\Omega\subseteq\R^3$ a bounded, smooth domain,
 and let~$g\colon\Omega\to\mcS_0$ be a bounded, Lipschitz map.
 For any~$Q\in H^1_g(\Omega, \, \mcS_0)$ and~$\sigma\in (0, \, 1)$, there exists 
 a bounded, Lipschitz map~$Q_\sigma\colon\overline{\Omega}\to\mcS_0$
 that satisfies 
 the following properties:
 \begin{gather}
  Q_\sigma = g \qquad \textrm{on } \partial\Omega \\
  \norm{Q_\sigma}_{L^\infty(\Omega)}
   \lesssim \sigma^{-1/2}\left(\norm{Q}_{H^1(\Omega)}
   + \norm{g}_{L^\infty(\Omega)}\right) \label{conv_Linfty-bdd} \\
  \norm{\nabla Q_\sigma}_{L^\infty(\Omega)}
   \lesssim \sigma^{-3/2} \left(\norm{Q}_{H^1(\Omega)} 
   + \norm{g}_{W^{1,\infty}(\Omega)}\right) \label{conv_Lipschitz-bdd} \\
  \norm{Q - Q_\sigma}_{L^2(\Omega)} \lesssim 
   \sigma\norm{Q}_{H^1(\Omega)} \label{conv_conv-bdd} \\
  \norm{\nabla Q - \nabla Q_\sigma}_{L^2(\Omega)} 
   \to 0 \qquad \textrm{as } \sigma\to 0. \label{conv_H1-bdd}
 \end{gather}
\end{lemma}
\begin{proof} 
 Since~$\Omega$ is bounded and smooth, we can extend~$g$
 to a bounded, Lipschitz map~$\R^3 \to\mcS_0$, still denoted~$g$
 for simplicity, in such a way that
 $\norm{g}_{L^\infty(\R^3)} \lesssim \norm{g}_{L^\infty(\Omega)}$,
 $\norm{\nabla g}_{L^\infty(\R^3)} \lesssim \norm{\nabla g}_{L^\infty(\Omega)}$.
 Let~$P := Q - g$. Then, $P\in H^1_0(\Omega, \, \mcS_0)$,
 and we extend~$P$ to a new map~$P\in H^1(\R^3, \, \mcS_0)$
 by setting~$P := 0$ on~$\R^3\setminus\Omega$. By applying 
 Lemma~\ref{lemma:convolution} to~$P$,
 we construct a family of smooth maps~$(P_\sigma)_{\sigma>0}$
 that satisfies~\eqref{conv_Lipschitz}--\eqref{conv_local}.
 We define
 \[
  \tilde{P}_\sigma(x) := \min\left(1, \, 
  \sigma^{-1}\dist(x, \, \partial\Omega) \right) P_\sigma(x)
  \qquad \textrm{for } x\in\Omega.
 \]
 The map~$\tilde{P}_\sigma\colon\overline{\Omega}\to\mcS_0$ 
 is bounded, Lipschitz, and $\tilde{P}_\sigma = 0$ on~$\partial\Omega$.
 We claim that~$\tilde{P}_\sigma$ satisfies 
 \eqref{conv_Linfty-bdd}--\eqref{conv_H1-bdd} with $g\equiv 0$; the lemma will follow
 by taking~$Q_\sigma := \tilde{P}_\sigma + g$. 
 First, we note that \eqref{conv_Linfty-bdd} is a consequence of the extension of  $P$ to the whole $\R^3$ and \eqref{conv_Lipschitz} together with the Gagliardo-Nirenberg-Soboleve inequality. Then we check~\eqref{conv_Lipschitz-bdd}. 
 Clearly~$|\tilde{P}_\sigma|\leq |P_\sigma|$. Using the chain rule,
 and keeping in mind that the function
 $\dist(\cdot, \, \partial\Omega)$ is~$1$-Lipschitz, we see that
 \begin{equation} \label{conv2}
  |\nabla\tilde{P}_\sigma| \leq |\nabla P_\sigma| + \sigma^{-1}|P_\sigma|
  \qquad \textrm{a.e. on } \Omega
 \end{equation}
 and~\eqref{conv_Lipschitz-bdd} follows, with the 
 help of~\eqref{conv_Lipschitz}. 
 
 We pass to the proof 
 of~\eqref{conv_conv-bdd}. Let~$\Gamma_\sigma := \{x\in\R^3\colon 
 \dist(x, \, \partial\Omega)<\sigma\}$. Since~$\partial\Omega$ 
 is a compact, smooth manifold, for sufficiently small~$\sigma$
 the set~$\Gamma_\sigma$ is diffeomorphic to the 
 product~$\partial\Omega\times(-\sigma, \, \sigma)$.
 We identify~$\Gamma_\sigma\simeq\partial\Omega\times(-\sigma, \, \sigma)$
 and denote the variable in~$\Gamma_\sigma$ 
 as~$x = (y, \, t)\in\partial\Omega\times(-\sigma, \, \sigma)$.
 We apply Poincar\'e inequality to the map~$P$
 on each slice~$\{y\}\times(-\sigma, \, \sigma)$:
 \[
  \int_{-\sigma}^\sigma \abs{P(y, \, t)}^2\d t
  = \int_0^\sigma \abs{P(y, \, t) - P(y, \, 0)}^2 \d t
 \lesssim \sigma^2 \int_0^\sigma \abs{\partial_t P(y, \, t)}^2 \d t.
 \]
 By integrating with respect to~$y\in\partial\Omega$, we obtain
 \begin{equation} \label{conv3}
  \norm{P}_{L^2(\Gamma_\sigma)} \lesssim 
  \sigma\norm{\nabla P}_{L^2(\Gamma_\sigma)}
 \end{equation}
 and hence,
 \[
  \begin{split}
   \|\tilde{P}_\sigma - P_\sigma\|_{L^2(\Omega)}
  \lesssim \norm{P_\sigma}_{L^2(\Gamma_\sigma)}
 \lesssim \norm{P}_{L^2(\Gamma_\sigma)} + \norm{P - P_\sigma}_{L^2(\Omega)}
   \stackrel{\eqref{conv_conv}, \eqref{conv3}}{\lesssim} 
   \sigma \norm{\nabla P}_{L^2(\Omega)} \! .
  \end{split}
 \]
 Finally, let us prove~\eqref{conv_H1-bdd}.
 Combining~\eqref{conv_local} and~\eqref{conv3}, we deduce
 \begin{equation} \label{conv4}
  \norm{P_\sigma}_{L^2(\Gamma_{\sigma})} 
  \leq \norm{P}_{L^2(\Gamma_{2\sigma})} 
  \lesssim \sigma\norm{\nabla P}_{L^2(\Gamma_{2\sigma})} \! .
 \end{equation}
 Therefore, we have
 \[
  \begin{split}
   \|\nabla\tilde{P}_\sigma - \nabla P_\sigma\|_{L^2(\Omega)}
   &\leq \|\nabla\tilde{P}_\sigma\|_{L^2(\Gamma_\sigma)} 
     + \norm{\nabla P_\sigma}_{L^2(\Gamma_\sigma)} \\
   &\stackrel{\eqref{conv2}}{\lesssim} \norm{\nabla P_\sigma}_{L^2(\Gamma_\sigma)}
     + \sigma^{-1} \norm{P_\sigma}_{L^2(\Gamma_{\sigma})} \\
   &\stackrel{\eqref{conv4}}{\lesssim} \norm{\nabla P}_{L^2(\Gamma_{2\sigma})}
     + \norm{\nabla P_\sigma - \nabla P}_{L^2(\Omega)}
  \end{split}
 \]
 and both terms in the right-hand side converge
 to zero as~$\sigma\to 0$, due to~\eqref{conv_H1}.
\end{proof}

\begin{lemma} \label{lemma:recovery}
 For any~$Q\in H^1_g(\Omega, \, \mcS_0)$, there exists 
 a sequence~$(Q_\eps)_{\eps>0}$ in~$H^1_g(\Omega, \, \mcS_0)$
 that converges $H^1(\Omega)$-strongly to~$Q$ and satisfies
 \[
  \abs{J_\eps[Q_\eps] - J_0[Q]} \lesssim \eps^{1/4}
  \left(\norm{Q}^4_{H^1(\Omega)} + 1\right)
 \]
 (the functionals~$J_\eps$, $J_0$ are defined 
 in~\eqref{Jeps}, \eqref{J0} respectively).
 The constant implied in front of the right-hand side 
 depends on the~$L^\infty(\partial\Omega)$-norms of~$g$ and~$\nabla g$,
 as well as~$\Omega$, $f_s^j$, $\mcP^j$, $R_*^j$ with $j\in\{1,\dots,J\}$.
\end{lemma}
\begin{proof}
 Let us fix a small~$\eps>0$. Let~$\beta$ 
 be a positive parameter, to be chosen later,
 and let~$Q_\eps := Q_{\eps^\beta}\in H^1_g(\Omega, \, \mcS_0)$
 be the Lipschitz map given by Lemma~\ref{lemma:conv-bdd}. We have
 \begin{equation} \label{recovery0}
  \begin{split}
   \abs{J_\eps[Q_\eps] - J_0[Q]} \leq 
    |J_\eps[Q_\eps] - \tilde{J}_\eps[Q_\eps]|
    + |\tilde{J}_\eps[Q_\eps] - J_0[Q_\eps]|
    + \abs{J_0[Q_\eps] - J_0[Q]}
  \end{split}
 \end{equation}
 where~$\tilde{J}_\eps$ is defined by~\eqref{Jtilde_eps}.
 We will estimate separately all the terms in the right-hand side.
 
 First, let us estimate the difference $J_\eps[Q_\eps] - \tilde{J}_\eps[Q_\eps]$.
 This can be achieved with the help of Lemma~\ref{lemma:Jeps}:
 \begin{equation} \label{recovery1}
  \begin{split}
   |J_\eps[Q_\eps] - \tilde{J}_\eps[Q_\eps]|
   &\stackrel{\eqref{Jtilde_eps_conv}}{\lesssim} 
    \eps^\alpha \left(\norm{Q_\eps}_{L^\infty(\Omega)}^3 + 1\right) 
    \norm{\nabla Q_\eps}_{L^\infty(\Omega)} \\
   &\stackrel{\eqref{conv_Linfty-bdd}, \eqref{conv_Lipschitz-bdd}}{\lesssim} 
    \eps^{\alpha - 3\beta} \left(\norm{Q}_{H^1(\Omega)}^4 + 1\right)
  \end{split}
 \end{equation}
 (here and througout the rest of the proof, the constant implied
 in front of the right-hand side may depend on the $L^\infty$-norms of~$g$ and~$\nabla g$).
 
 As for the second term, $\tilde{J}_\eps[Q_\eps] - J_0[Q_\eps]$,
 we write~$\tilde{J}_\eps$ in the form~\eqref{Jtilde_eps_repr}
 and we re-write~$J_0$ using~\eqref{f_hom}, \eqref{J0}:
 \begin{equation} \label{recovery1.5}
  \begin{split}
   |\tilde{J}_\eps[Q_\eps] - J_0[Q_\eps]|
   &\leq \sum_{j=1}^J \abs{\int_{\Omega} f_{hom}^j(Q_\eps, \, x) \, \d\mu_\eps^j
   - \int_{\Omega} f_{hom}^j(Q_\eps, \, x) \, \xi^j \, \d x} \\
   &\stackrel{\eqref{flat}}{\leq} \mathbb{F}_\eps
   \sum_{j=1}^J \left(\|\nabla (f_{hom}^j(Q_\eps, \, \cdot))\|_{L^\infty(\Omega)} 
   + \|f_{hom}^j(Q_\eps, \, \cdot)\|_{L^\infty(\Omega)}\right) \! .
  \end{split}
 \end{equation}
 To estimate the terms at the right-hand side, we apply 
 Lemma~\ref{lemma:fhom} and Lemma~\ref{lemma:conv-bdd}:
 \[
  \begin{split}
   \|\nabla (f_{hom}^j(Q_\eps, \, \cdot))\|_{L^\infty(\Omega)} 
   &\stackrel{\eqref{fhomj-Lipschitz}}{\lesssim} 
    \left(\norm{Q_\eps}_{L^\infty(\Omega)}^3 + 1 \right)
    \norm{\nabla Q_\eps}_{L^\infty(\Omega)} \\
   &\stackrel{\eqref{conv_Lipschitz-bdd}}{\lesssim} 
    \eps^{- 3\beta} \left(\norm{Q}_{H^1(\Omega)}^4 + 1\right) \! ,
  \end{split}
 \]
 and
 \[
  \begin{split}
   \|f_{hom}^j(Q_\eps, \, \cdot)\|_{L^\infty(\Omega)} 
   &\stackrel{\eqref{fhomj-Lipschitz}}{\lesssim} 
    \norm{Q_\eps}_{L^\infty(\Omega)}^4 + 1 
   \stackrel{\eqref{conv_Lipschitz-bdd}}{\lesssim} 
    \eps^{- 2\beta} \left(\norm{Q}_{H^1(\Omega)}^4 + 1\right) \! .
  \end{split}
 \]
 Injecting these inequalities into~\eqref{recovery1.5}, and using 
 that~$\mathbb{F}_\eps\lesssim\eps$ by Assumption~\eqref{flat},
 we obtain
 \begin{equation} \label{recovery2}
  \begin{split}
   |\tilde{J}_\eps[Q_\eps] - J_0[Q_\eps]|
   &\lesssim \eps^{1- 3\beta} \left(\norm{Q}_{H^1(\Omega)}^4 + 1\right) \!.
  \end{split}
 \end{equation}
 Finally, the term~$J_0[Q_\eps] - J_0[Q]$. We apply 
 Lemma~\ref{lemma:fhom} and the H\"older inequality:
 \begin{equation*} 
  \begin{split}
   \abs{J_0[Q_\eps] - J_0[Q]} 
   &\leq \int_{\Omega} \abs{f_{hom}(Q_\eps, \, \cdot) 
    - f_{hom}(Q, \, \cdot)}  \\  
   &\stackrel{\eqref{fhom-Lipschitz}}{\leq} 
    \int_{\Omega} \left(\abs{Q}^3 + \abs{Q_\eps}^3 + 1\right) \abs{Q - Q_\eps} \\
   &\lesssim \left(\norm{Q}_{L^6(\Omega)}^3 
    + \norm{Q_\eps}_{L^6(\Omega)}^3 + 1 \right)
    \norm{Q - Q_\eps}_{L^2(\Omega)}
  \end{split}
 \end{equation*}
 The sequence~$Q_\eps$ is bounded in~$L^6(\Omega)$,
 thanks to Sobolev embedding and to Lemma~\ref{lemma:conv-bdd}.
 Therefore, 
 \begin{equation} \label{recovery3}
  \begin{split}
   \abs{J_0[Q_\eps] - J_0[Q]} 
   &\stackrel{\eqref{conv_conv-bdd}}{\lesssim}
    \eps^\beta \norm{Q}_{H^1(\Omega)}^4 + \eps^\beta \norm{Q}_{H^1(\Omega)} \! .
  \end{split}
 \end{equation}
 Combining \eqref{recovery0}, \eqref{recovery1}, 
 \eqref{recovery2} and~\eqref{recovery3}, we deduce
 \[
  \abs{J_\eps[Q_\eps] - J_0[Q]} \lesssim 
  \eps^{\min(\alpha - 3\beta, \, 1 - 3\beta, \, \beta)}
  \left(\norm{Q}^4_{H^1(\Omega)} + 1\right) \! .
 \]
 Keeping into account that~$\alpha > 1$, we see that 
 the optimal choice of~$\beta$ is~$\beta = 1/4$, and the lemma follows.
\end{proof}

\begin{lemma} \label{lemma:lim_volume}
 Let $Q_\eps\in H^1(\Omega_\eps, \, \mcS_0)$
 be a family of maps, such that $E_\eps Q_\eps\to Q$
 strongly in~$H^1(\Omega)$, as~$\eps\to 0$. Then, 
 \begin{equation*}
  \limsup_{\eps\to 0} \int_{\Omega_\eps}  
   \left(f_e(\nabla Q_\eps) + f_b(Q_\eps)\right)\d x
  \leq \int_{\Omega} \left(f_e(\nabla Q) + f_b(Q)\right)\d x. 
 \end{equation*}
\end{lemma}
\begin{proof}
 By Sobolev embedding, we have $E_\eps Q_\eps\to Q$ strongly in~$L^6(\Omega)$.
 Then, up to extraction of a non-relabelled subsequence,
 we find functions~$h_e\in L^2(\Omega)$, $h_b\in L^6(\Omega)$ such that
 \begin{equation} \label{lim_volume1}
  \abs{\nabla(E_\eps Q_\eps)} \leq h_e, 
  \quad \abs{E_\eps Q_\eps} \leq h_b \qquad \textrm{a.e. on } \Omega, 
  \textrm{ for any } \eps.
 \end{equation}
 Let~$\chi_\eps$ be the  the indicator function 
 of~$\Omega_\eps$ (i.e., $\chi_\eps := 1$ on~$\Omega_\eps$ and
 $\chi_\eps := 0$ elsewhere). Thanks to~\eqref{hp:fe}, \eqref{hp:fb}
 and to~\eqref{lim_volume1}, we have
 \[
  \left(f_e(\nabla Q_\eps) + f_b(Q_\eps)\right) \chi_\eps
  \lesssim h_e^2 + h_b^6 + 1 \in L^1(\Omega)
  \qquad \textrm{a.e. on } \Omega, \textrm{ for any } \eps.
 \]
 Moreover, since $|\mcP_\eps|\lesssim \eps^{3\alpha - 3}\to 0$, 
 $\chi_{\eps}$ converges to~$1$ strongly in~$L^1(\Omega)$
 and we may extract a further subsequence so to have
 $\chi_\eps\to 1$ a.e. Then, the lemma
 follows from Lebesgue's dominated converge theorem.
\end{proof}

\begin{proof}[Proof of Theorem~\ref{th:gamma}]
 Let~$Q_\eps\in H^1_g(\Omega_\eps, \, \mcS_0)$, for~$\eps>0$,
 be a family of maps such that $\sup_\eps\mcF_{\eps,\gamma}(Q)<+\infty$.
 We first extract a (non-relabelled) subsequence~$\eps\to 0$, so that
 $\limsup_{\eps\to 0}\mcF_{\eps,\gamma}(Q_\eps)$ is achieved as a limit;
 this allows us to pass freely to subsequences, in what follows.
 Thanks to~\eqref{hp:fe}, \eqref{hp:fb}, \eqref{hpgamma:fs},
 we have $\sup_\eps\|Q_\eps\|_{L^2(\Omega_\eps)}<+\infty$.
 By Lemma~\ref{lem:extension}, there is a (non-relabelled) subsequence
 and~$Q_0\in H^1_g(\Omega, \, \mcS_0)$ such that $E_\eps Q_\eps\rightharpoonup Q_0$
 weakly in~$H^1(\Omega)$. By Proposition~\ref{prop:liminf},
 there holds
 \[
  \widetilde{\mcF}(Q) \leq \liminf_{\eps\to 0} 
  \int_{\Omega_\eps}\left(f_e(\nabla Q_\eps) + f_b(Q_\eps)\right)
  \leq \liminf_{\eps\to 0} \mcF_{\eps,\gamma}(Q_\eps)
 \]
 and
 \[
  J_0[Q_0] = \lim_{\eps\to 0} J_\eps[Q_\eps] \leq 
  \limsup_{\eps\to 0} \eps^\gamma \mcF_{\eps,\gamma}(Q_\eps) = 0,
 \]
 so~$Q_0$ belongs to the class~$\mcA$ defined by~\eqref{A}.
 Thus, Statement~(i) is proved.
 
 We now prove Statement~(ii). Let~$Q_0\in H^1_g(\Omega, \, \Omega)$
 be fixed. We can suppose without loss of generality that~$Q\in\mcA$,
 otherwise the statement is trivial. Due to Lemma~\ref{lemma:recovery},
 there is a sequence~$\tilde{Q}_\eps\in H^1_g(\Omega, \, \mcS_0)$
 such that $\tilde{Q}_\eps\to Q_0$ strongly in~$H^1(\Omega)$ and
 \begin{equation} \label{gamma1}
  \abs{J_\eps[\tilde{Q}_\eps]} = \eps^{3-2\alpha} \abs{\sum_{j=1}^J
  \int_{\partial\mcP^j} f_s^j(\tilde{Q}_\eps, \, \nu) \,\d\sigma}
  \lesssim \eps^{1/4} \left(\norm{Q}_{L^4(\Omega)}^4 + 1\right) \! .
 \end{equation}
 Let~$Q_\eps := \tilde{Q}_{\eps|\Omega_\eps}$. By Lemma~\ref{lem:extension},
 $E_\eps Q_\eps\to Q_0$ strongly in~$H^1(\Omega)$. Using Lemma~\ref{lemma:lim_volume}
 and~\eqref{gamma1}, and recalling that~$\gamma < 1/4$, we conclude that
 \[
  \limsup_{\eps\to 0} \mcF_{\eps,\gamma}(Q_\eps) =
  \limsup_{\eps\to 0} \int_{\Omega_\eps}\left(f_e(\nabla Q_\eps) + f_b(Q_\eps)\right)
  +\underbrace{\limsup_{\eps\to 0} \eps^{-\gamma} J_\eps[Q_\eps]}_{=0, \textrm{ by \eqref{gamma1}}}
  \leq \widetilde{\mcF}(Q_0),
 \]
 so the proof is complete. 
\end{proof}

\section*{Acknowledgements}

The work of A.Z. is supported by the Basque Government through the BERC 2018-2021
program, by Spanish Ministry of Economy and Competitiveness MINECO through BCAM
Severo Ochoa excellence accreditation SEV-2017-0718 and through project MTM2017-82184-R
funded by (AEI/FEDER, UE) and acronym ``DESFLU''. 



\end{document}